\documentclass[notheoremnums]{dpreprint}

\newtheorem{theorem}{Theorem}
\newtheorem{lemma}{Lemma}[section]
\newtheorem{corollary}[lemma]{Corollary}

\newtheorem{proposition}[lemma]{Proposition}

\theoremstyle{definition}
\newtheorem{definition}[lemma]{Definition}
\newtheorem{example}[lemma]{Example}
\newtheorem{remark}[lemma]{Remark}

\newtheorem*{claim}{Claim}

\usepackage{xcolor}
\usepackage{pgf, tikz}
\usetikzlibrary{arrows, automata}

\newcommand{\dens}{\overline{d}}

\newcommand{\len}[1]{| #1 |}

\newtheorem*{Theorem2}{Theorem \ref{disc}}
\newtheorem*{Theorem3}{Theorem \ref{example}}

\begin{document}

\dtitle{Low complexity subshifts have discrete spectrum}

\dauthorone[Darren~Creutz]{Darren Creutz}{creutz@usna.edu}{US Naval Academy}{}
\dauthortwo[Ronnie~Pavlov]{Ronnie Pavlov}{rpavlov@du.edu}{University of Denver}{The second author gratefully acknowledges the support of a Simons Foundation Collaboration Grant.}
\datewritten{\today}


\dabstract{We prove results about subshifts with linear (word) complexity, meaning that $\limsup \frac{p(n)}{n} < \infty$, where for every $n$, $p(n)$ is the number of $n$-letter words appearing in sequences in the subshift. Denoting this limsup by $C$, we show that when $C < \frac{4}{3}$, the subshift has discrete spectrum, i.e. is measurably isomorphic to a rotation of a compact abelian group with Haar measure. We also give an example with $C = \frac{3}{2}$ which has a weak mixing measure. This partially answers an open question of Ferenczi, who asked whether $C = \frac{5}{3}$ was the minimum possible among such subshifts; our results show that the infimum in fact lies in $[\frac{4}{3}, \frac{3}{2}]$.
All results are consequences of a general S-adic/substitutive structure proved when $C < \frac{4}{3}$.}

\makepreprint

\dsectionstar{Introduction}



The main objects of study in symbolic dynamics are \textbf{subshifts}, which are dynamical systems defined by a finite alphabet 
$\mathcal{A}$, a closed shift-invariant set of sequences $X \subset \mathcal{A}^\mathbb{Z}$, and the left-shift map $\sigma$. We sometimes speak of subshifts as measure-theoretic dynamical systems by associating a measure $\mu$; in this case $\mu$ is always assumed to be a Borel probability measure invariant under $\sigma$. One of the most basic ways to measure the `size' of a subshift $X$ is the \textbf{word complexity function} $p(n)$, which measures the number of finite words of length $n$ which appear within points of $X$. In addition to being intimately connected with the fundamental notion of topological entropy (the entropy $h(X)$ is just the exponential growth rate of $p(n)$ when $p(n)$ grows exponentially), many recent works prove that slow growth of $p(n)$ forces various strong structural properties of $X$. 

The well-known Morse-Hedlund theorem \cite{HM} implies that if $X$ is infinite, then $p(n) \geq n+1$ for all $n$. There are subshifts which achieve this minimal value (i.e. $p(n) = n+1$ for all $n$), which are called \textbf{Sturmian} subshifts. We do not give a full treatment here, but briefly say that Sturmian subshifts are defined by symbolic codings of orbits for irrational circle rotations, and in fact are measure-theoretically isomorphic to these rotations (associated with Lebesgue measure). 

Slightly above the minimum possible complexity is the property of linear complexity, meaning that $\limsup p(n)/n = C < \infty$. This implies a great deal about $X$; a full list is beyond this work, but we list a few such results here. In the following, $X$ is \textbf{transitive} when there exists $x \in X$ whose \textbf{orbit} $\{\sigma^n x\}$ is dense in $X$, and 
\textbf{minimal} when every $x \in X$ has dense orbit.

\begin{enumerate}
\item If $X$ is transitive, then the number of ergodic measures on $X$ is bounded from above by $\lfloor C \rfloor$. 
If $C < 3$, then in fact there is only one $\sigma$-invariant measure on $X$, in which case $X$ is said to be
\textbf{uniquely ergodic}. (\cite{boshernitzan}, \cite{DMP})
\item For all $X$, the number of nonatomic generic measures on $X$ is bounded from above by $\lfloor C \rfloor$ (\cite{CK2})
\item If $X$ is minimal, then the automorphism group of $X$ is virtually $\mathbb{Z}$ (in particular, there are at most $\lfloor C \rfloor$ cosets once one mods out by the shift action) (\cite{CK1}, \cite{DDMP})
\item If $X$ is minimal, then $X$ has finite topological rank (\cite{DDMP2})
\item $X$ cannot have any nontrivial strongly mixing measure (\cite{ferenczirank})
\item If $X$ is transitive and $C < 1.5$, then $X$ is minimal (\cite{ormespavlov})
\end{enumerate}

(In fact, the weaker condition $\liminf p(n)/n < \infty$ is sufficient for some of the structure above, but as our results don't involve this quantity, we don't comment on it further here.)
The final item above is one of surprisingly few results proved about subshifts with $C$ close to $1$, and understanding more about the structure of such shifts was a main motivation of this work. In a sense, we show that for $C$ sufficiently close to $1$, a subshift must have structure more and more similar to the Sturmian subshifts, which achieve minimal possible complexity. Recall that Sturmian subshifts are measure-theoretically isomorphic to a (compact abelian) group rotation; this property is called 
\textbf{discrete spectrum}. In fact this property is equivalent to $L^2(X)$ being spanned by the measurable eigenfunctions of $\sigma$ (i.e. $f$ for which $f(\sigma x) = \lambda f(x)$ for some $\lambda$). When $X$ has no eigenfunctions at all, it is said to be 
\textbf{weak mixing}, which is in a sense an opposite property to discrete spectrum.

Ferenczi (\cite{ferenczirank}) proved that the 
property of \textbf{strong mixing} (which means that $\mu(A \cap \sigma^{-n} B) \rightarrow \mu(A) \mu(B)$ for all measurable $A,B$) cannot hold for any nontrivial measure on a linear complexity subshift. He also gave an example of $X$ with a strongly mixing measure and $p(n)$ quadratic and asked whether this complexity was the lowest possible. This was proved not to be the case in \cite{Creutz2022} and \cite{CPR}, which provided examples first on the order of $n \log n$, and then below any possible superlinear growth rate, establishing linear complexity as the `threshold' for existence of such a measure. In a different work, Ferenczi (\cite{ferenczichacon}) examined the same question for weakly mixing measures, where it is known that linear complexity can occur via the well-known Chacon subshift. He there gave an example of $X$ with a nontrivial weakly mixing measure and $C = 5/3$, and again asked whether this was minimal. This was shown not to be the case in \cite{Creutz2022b}, where examples were given of $C$ arbitrarily close to (but above) $3/2$. 

Our main results are the following.


\begin{Theorem2}
If $X$ is an infinite transitive subshift with $\limsup \frac{p(q)}{q} < \frac{4}{3}$, then $X$ is uniquely ergodic with unique measure which has discrete spectrum.
\end{Theorem2}

\begin{Theorem3}
There exists an infinite transitive subshift $X$ which is uniquely ergodic, has unique measure which is weak mixing, and 
for which $\limsup \frac{p(q)}{q} = \frac{3}{2}$.
\end{Theorem3}

In \cite{Creutz2022b}, it was also suggested that perhaps a subshift $X$ having a nontrivial weakly mixing measure forces $\limsup \frac{p(q)}{q} > \frac{3}{2}$; Theorem~\ref{example} answers this negatively.  In fact,
the examples from Theorem~\ref{example} satisfy $\lim p(q) - 1.5q = -\infty$,
in contrast to Theorem C from \cite{Creutz2022b}, which showed that for rank-one subshifts, even total ergodicity implies 
$\limsup p(q) - 1.5q = \infty$.  The examples also satisfy $\liminf \frac{p(q)}{q} = 1$ and for any $f(q) \to \infty$, there exist examples such that $p(q) < q + f(q)$ infinitely often.

The proof of Theorem \ref{disc} depends on proving a substitutive structure for subshifts with $C < \frac{4}{3}$.
In fact, for any $C < 2$, Corollary 5.28 from \cite{ps} already implies that $X$ can be generated by a sequence of substitutions $\tau_k$ on the alphabet $\{0,1\}$; this is known as having \textbf{alphabet rank two}. Similar results from
\cite{DDMP2} prove that even $\liminf p(n)/n < \infty$ implies finite alphabet rank. However, in general it is not so easy to prove dynamical properties of a subshift purely from such a structure; the key of our arguments is that when $C$ is closer to $1$, these substitutions come from a very restricted class.  We would like to note that subshifts with $p(n) \leq 4n/3 + 1$ were also studied in \cite{aberkane}, where the author proved a substitutive structure and gave some interesting examples.

Specifically, our Proposition~\ref{decomp} shows that any such subshift is induced by a sequence of substitutions of the form
$\tau_{m_k,n_k}: 0 \mapsto 0^{m_k-1} 1, 1 \mapsto 0^{n_k-1} 1$ where $n \leq 2m$ for $m > 1$ and $n \leq 3$ for $m = 1$. 
This is related to the well-known Pisot conjecture for subshifts, which states that a subshift generated by iterating a single substitution $\tau$ should have discrete spectrum if the associated matrix (in which the $(a,b)$ entry is the number of occurrences of $b$ in $\tau(a)$) has largest eigenvalue which is a Pisot number (i.e. a complex number with modulus greater than $1$ all of whose conjugates have modulus less than $1$). 

The Pisot conjecture has been proved in some settings, including when $|\mathcal{A}| = 2$ (\cite{MR1947456}, \cite{2pisot}) and whenever the so-called balanced pair algorithm terminates (\cite{balpair}). Our proof of Theorem~\ref{disc} is in fact based on this algorithm.

In our case, the substitutive structure comes from a sequence of substitutions and not a single one; this is sometimes called the S-adic Pisot conjecture, based on the often-used term `S-adic' (among other references, see \cite{sadic}) to refer to sequences obtained by a sequence of substitutions on a fixed alphabet. This is much more difficult. The strongest result is due to \cite{bertheetal}, which is too long to state formally here, but which proves discrete spectrum in a fairly general S-adic setting. They do require, however, that the sequence of substitutions $(\tau_n)$ be recurrent, meaning that for every $k$, there exists $L$ so that $\tau_i = \tau_{i+L}$ for $1 \leq i \leq k$. 

We cannot enforce any such condition on our substitutions, as it's quite possible to have low complexity for $\tau_{m_k, n_k}$ all distinct (for instance, consider Sturmian subshifts, which can be generated by an infinite sequence of distinct substitutions if the digits of its continued fraction expansion are distinct). Nevertheless, due to the extremely simple form of $\tau_{m_k, n_k}$ (in which both $0$ and $1$ are mapped to words of the form $0^i 1$), we are able to prove discrete spectrum. 

We note that indeed our substitutive structure is in some sense Pisot; the associated matrix for $\tau_{m,n}$ is $\left(\begin{smallmatrix}m-1 & 1\\ n-1 & 1 \end{smallmatrix}\right)$, whose eigenvalues are $\frac{\sqrt{m^2 + 4(n-m)}\pm m}{2}$. This matrix is Pisot when $m < n \leq 2m$. Our Proposition~\ref{decomp} implies $m < n \leq 2m$, with the possible exception $m = 1, n = 3$. Though this substitution is not Pisot, Proposition~\ref{decomp} implies that when it occurs, the previous substitution has $n = m+1$, and the composition of those substitutions has matrix $\left(\begin{smallmatrix}0 & 1\\ 2 & 1 \end{smallmatrix}\right) \left(\begin{smallmatrix}m-1 & 1\\ m & 1 \end{smallmatrix}\right) = \left(\begin{smallmatrix}m & 1\\ 3m-2 & 3 \end{smallmatrix}\right)$, which is always Pisot. 

One of course should not expect that simply assuming each $\tau_i$ to be Pisot should guarantee discrete spectrum; informally, if the second eigenvalues have moduli each less than $1$ but which converge to $1$ quickly, then the `average behavior' will be that of a non-Pisot number. This is essentially the construction of our example from Theorem~\ref{example}, which not only does not have discrete spectrum, but is weak mixing (i.e. has no eigenvalue at all). 

\section{Definitions and preliminaries}\label{defs}


Let $\mathcal{A}$ be a finite subset of $\mathbb{Z}$; the \textbf{full shift} is the set $\mathcal{A}^\mathbb{Z}$ associated with the product topology. We use $\sigma$ to denote the left shift homeomorphism on $\mathcal{A}^\mathbb{Z}$. 
A \textbf{subshift} is a closed $\sigma$-invariant subset $X \subset \mathcal{A}^\mathbb{Z}$. The \textbf{orbit} of $x \in X$ is the 
set $\{\sigma^n x\}_{n \in \mathbb{Z}}$.
 A subshift $X$ is \textbf{transitive} when it is the closure of the orbit of a single sequence 
$x$, and \textbf{minimal} when it is the closure of the orbit of every $x \in X$. For a minimal subshift $X$, in a slight abuse of notation, we sometimes refer to $X$ as the orbit closure of a one-sided sequence $y \in \mathcal{A}^\mathbb{N}$; this simply means that 
$X$ is the orbit closure of a two-sided sequence $x \in X$ containing $y$.

A \textbf{word} is any element of $\mathcal{A}^n$ for some $n \in \mathbb{N}$, referred to as its \textbf{length} and denoted by
$\len{w}$. We denote $\mathcal{A}^* = \bigcup_{n \geq 1} \mathcal{A}^n$. We represent the concatenation of words $w_1, w_2, \ldots, w_n$ by $w_1 w_2 \ldots w_n$. 

The \textbf{language} of a subshift $X$ on $\mathcal{A}$, denoted $L(X)$, is the set of all finite words appearing as subwords of points in $X$. For any $n \in \mathbb{N}$, we denote $L_n(X) = L(X) \cap \mathcal{A}^n$, the set of $n$-letter words in $L(X)$.
For a subshift $X$, the \textbf{word complexity function} of $X$ is defined by $p(n) := |L_{n}(X)|$. 
For a subshift $X$ and word $w \in L(X)$ we denote by $[w]$ the clopen subset in $X$ consisting of all $x \in X$ such that $x_{0}\ldots x_{|w|-1} = w$. 

One way to generate subshifts is via substitutions. 
A \textbf{substitution} (sometimes called a morphism) is a map 
$\tau: \mathcal{A} \rightarrow \mathcal{B}^*$ for finite alphabets $\mathcal{A}$ and $\mathcal{B}$. An example is the well-known Thue-Morse substitution $\{ 0, 1 \} \to \{ 0, 1 \}^{*}$ given by $0 \mapsto 01$ and $1 \mapsto 10$.

Substitutions can be composed when viewed as homomorphisms on the monoid of words under composition, i.e. if $\tau: \mathcal{A} \rightarrow \mathcal{B}^*$ and
$\rho: \mathcal{B} \rightarrow \mathcal{C}^*$, then $\rho \circ \tau: \mathcal{A} \rightarrow \mathcal{C}^*$ can be defined by
$(\rho \circ \tau)(a) = \rho(b_1) \rho(b_2) \ldots \rho(b_k)$, where $\tau(a) = b_1 \ldots b_k$.  

When a sequence of substitutions $\tau_k: \mathcal{A} \rightarrow \mathcal{A}^*$ shares the same alphabet, and when there exists $a \in \mathcal{A}$ for which $\tau_k(a)$ begins with $a$ for all $k$, clearly
$(\tau_1 \circ \cdots \circ \tau_k)(a)$ is a prefix of $(\tau_1 \circ \cdots \circ \tau_{k+1})(a)$ for all $k$. In this situation one may then speak of the (right-infinite) limit of $(\tau_1 \circ \cdots \circ \tau_k)(a)$. For example, if all $\tau_k$ are equal to the Thue-Morse substitution and $a = 0$, the limiting sequence is $.0110100110010110 \ldots$, and the orbit closure of this sequence is called the Thue-Morse substitution subshift. 

For any subshift $X$, there is a convenient way to represent the $n$-language and possible transitions between words in points of $X$ by a directed graph called the Rauzy graph.

\begin{definition}
For a subshift $X$ and $n \in \mathbb{N}$, the \textit{$n$th Rauzy graph of $X$} is the directed graph $G_{X, n}$ with vertex
set $L_{n}(X)$, and directed edges from $w_1 \ldots w_{n}$ to $w_2 \ldots w_{n+1}$ for all $w_1 \ldots w_{n+1} \in L_{n+1}(X)$.
\end{definition}

\begin{example}
If $X$ is the golden mean subshift consisting of bi-infinite sequences on $\{0,1\}$ without consecutive $1$s, and $n = 3$, then 
$G_{X, 3}$ is the following directed graph:\\

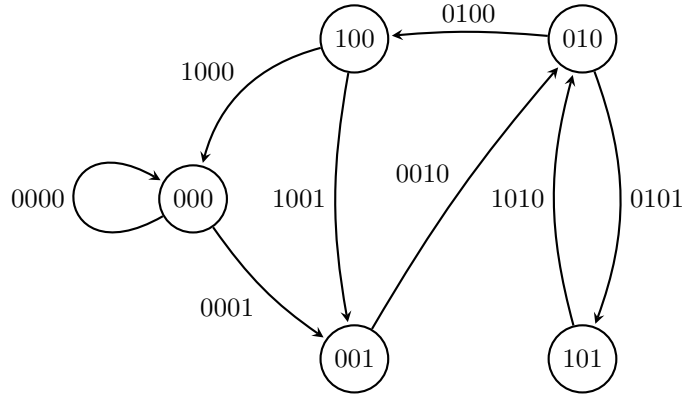
\begin{center}
\tikzset{every loop/.style={min distance=10mm,in=150,out=210,looseness=10}}
\begin{tikzpicture}[
            > = stealth, 
            shorten > = 1pt, 
            auto,
            node distance = 3cm, 
            thick 
        ]

        \tikzstyle{every state}=[
            draw = black,
            thick,
            fill = white,
            minimum size = 4mm
        ]

        \node[state] (000) {$000$};
        \node[state] (100) [above right of=000] {$100$};
        \node[state] (010) [right of=100] {$010$};
        \node[state] (001) [below right of=000] {$001$};
        \node[state] (101) [right of=001] {$101$};

        \path[->] (000) edge [loop left] node {0000} (000);
        \path[->] (000) edge [bend right=10] node[swap] {0001} (001);
        \path[->] (100) edge [bend right=10] node[swap] {1001} (001);
        \path[->] (100) edge [bend right=30] node[swap] {1000} (000);
        \path[->] (010) edge [bend right=5] node[swap] {0100} (100);
        \path[->] (001) edge [bend left=5] node {0010} (010);
        \path[->] (101) edge [bend left=15] node {1010} (010);
        \path[->] (010) edge [bend left=20] node {0101} (101);
    \end{tikzpicture}
    \end{center}
\end{example}

There is a natural association from bi-infinite paths on the Rauzy graph to sequences in $\mathcal{A}^{\mathbb{Z}}$; a sequence of vertices
$(v_k)$ corresponds to the sequence $x \in A^{\mathbb{Z}}$ defined by $x(k) \ldots x(k+n-1) = v_k$ for all $k$.
The main usage of the Rauzy graph is that every point of $X$ corresponds to a bi-infinite path in the Rauzy graph. However, the opposite is not necessarily true; if $X$ has restrictions/forbidden words of length greater than $n+1$, then there may be paths in the Rauzy graph whose associated sequences are not in $X$. However, when $X$ has low word complexity function, the set of paths in the Rauzy graph is sufficiently restrictive to give us useful information about (but not necessarily a complete description of) $X$.

We note that when $X$ is transitive, $G_{X,n}$ is strongly connected for all $n$, i.e. there is a path between any two vertices. Rauzy graphs are particularly useful for working with so-called left/right special words in $L(X)$.

\begin{definition}
A word $w \in L(X)$ is \textbf{left-special} (resp. \textbf{right-special}) if there exist $a \neq b \in \mathcal{A}$ so that
$aw, bw \in L(X)$ (resp. $wa, wb \in L(X)$). A word is \textbf{bi-special} if it is both left- and right-special.
\end{definition}

For a given $n$, the left- and right-special words in $L_n(X)$ correspond to vertices of $G_{X,n}$ with multiple incoming/outgoing 
edges respectively. When $G_{X,n}$ has relatively few such vertices, large portions of bi-infinite paths are `forced' in the sense that when such a path visits a vertex which is not right-special, there is only one choice for the following edge. 
Note that if $X$ contains no right-special words of some length $n$, then any edge of $G_{X,n}$ forces all subsequent edges, meaning that $G_{X,n}$ has only finitely many bi-infinite paths and $X$ is finite. Therefore every infinite subshift $X$ has right-special words of every length, and a similar argument shows that it has left-special words of every length as well.

A particularly simple case that we deal with repeatedly is when $p(n+1) - p(n) = 1$; this means that $G_{X, n}$ has exactly one more edge than the number of vertices, which means that it has a single vertex $r$ with two outgoing edges and a single vertex $\ell$ with two incoming edges ($\ell$ and $r$ may be the same vertex), which correspond to the unique right- and left-special words in $L_n(X)$. It's not hard to show that when $X$ is transitive and $p(n+1) - p(n) = 1$, the structure of the Rauzy graph $G_{X, n}$ must be a (possibly empty) path from $\ell$ to $r$ and two edge-disjoint paths from $r$ to $\ell$. 

We will frequently make use of the following standard lemma, essentially contained in \cite{HM}, for estimating word complexity.

\begin{lemma}\label{RSlem}
Let $X$ be a subshift on alphabet $\mathcal{A}$, for all $n$ let $RS_n(X)$ denote the set of right-special words of length $n$ in the language of $X$, and for all right-special $w$, let
$F(w)$ denote the set of letters which can follow $w$, i.e. $\{a \ : \ wa \in L(X)\}$. Then, for all $q > r$,
\[
p(q) = p(r) + \sum_{i = r}^{q-1} \sum_{w \in RS_i(X)} (|F(w)| - 1).
\]
\end{lemma}

\begin{proof}
Consider the map $f: L_{r+1}(X) \rightarrow L_r(X)$ obtained by removing the final letter, i.e. $f(wa) = w$. It's clear that $f$ is surjective and that $|f^{-1}(w)| = 1$ for $w$ which is not right-special and $|f^{-1}(w)| = |F(w)|$ for $w \in RS_r(X)$. The result for $q = r+1$ follows immediately, and the general case follows by induction.
\end{proof}

The following corollary is immediate.

\begin{corollary}\label{RScor}
If $X$ is an infinite subshift and $T \subset \mathbb{N}$ denotes the set of lengths $n$ for which $|RS_n(X)| > 1$, then for all $q > r$,
\[
p(q) \geq p(r) + (q-r) + |T \cap \{r, \ldots, q-1\}|.
\]
If $|RS_i(X)| \leq 2$ for all $m \leq i < n$ and $|F(w)| = 2$ for all right-special $w$ with lengths in $[r,q)$, then the inequality above is an equality.
\end{corollary}

\section{Structure of subshifts with \texorpdfstring{$C < \nicefrac{4}{3}$}{C < 4/3}}

As mentioned above, our results rely on a substitutive/S-adic structure for subshifts with sufficiently low complexity. The substitutions in question all have the same form. Namely, for all positive integers $m<n$, define the substitution
\[
\tau_{m,n}: 
\begin{cases} 0 \mapsto 0^{m-1} 1\\ 1\mapsto 0^{n-1} 1. \end{cases}
\]
When $m_1, \ldots, m_k$ and $n_1, \ldots, n_k$ are understood, we use the shorthand notation
\[
\rho_k = \tau_{m_1, n_1} \circ \cdots \circ \tau_{m_k, n_k}.
\]

\begin{proposition}\label{decomp}
If $X$ is an infinite transitive subshift with $\limsup \frac{p(q)}{q} < \frac{4}{3}$, then there exists a substitution 
$\pi: \{0,1\} \rightarrow \mathcal{A}^*$ where $\pi(0), \pi(1)$ begin with different letters and $|\pi(0)| < |\pi(1)| < 2|\pi(0)|$ and sequences $(m_k), (n_k)$ satisfying $0 < m_k < n_k$ so that $X$ is the orbit closure of 
\[
x^{(m_k), (n_k)} = \lim_{k} (\pi \circ \tau_{m_1, n_1} \circ \cdots \circ \tau_{m_k, n_k})(0) = \lim_k \pi(\rho_k(0)).
\]
In addition,
\begin{itemize}
\item $n_{k} \leq 2m_{k}$ whenever $m_{k} > 1$;
\item $n_{k} < 1.9m_{k}$ whenever $m_{k} > 4$;
\item $n_{k} \leq 3$ whenever $m_{k} = 1$; 
\item if $m_{k+1} = 1, n_{k+1} = 3$ then $n_{k} = m_{k} + 1$; and
\item every right-special word of length at least $\len{s(\pi(0))^{m_{1}-1}}$, where $s$ is the maximal common suffix of $(\pi(0))^{\infty}$ and $(\pi(0))^{\infty}\pi(1)$, is a suffix of a concatenation of $\pi(0)$ and $\pi(1)$.
\end{itemize}
\end{proposition}


\begin{definition}
A word $v$ is a \textbf{root} of $w$ if $\len{v} \leq \len{w}$ and $w$ is a suffix of the left-infinite word $v^{\infty}$.  The \textbf{minimal root} of $w$ is the shortest $v$ which is a root of $w$.  

\end{definition}

Every word $w$ has a unique minimal root since it is a root of itself (and all roots of $w$ are suffixes of $w$).

\begin{lemma}[\cite{Creutz2022} Lemma 5.7]\label{B}
If $w$ and $v$ are words with $\len{v} \leq \len{w}$ such that $wv$ has $w$ as a suffix then $v$ is a root of $w$.
\end{lemma}

\begin{lemma}[\cite{Creutz2022} Lemma 5.8]\label{B1}
If $uv = vu$ then $u$ and $v$ are powers of the same word, i.e.~there exists a word $v_{0}$ and integers $t,s > 0$ such that $u = v_{0}^{t}$ and $v = v_{0}^{s}$.
\end{lemma}

\begin{lemma}\label{s}
Let $u$ and $v$ be words with $\len{v} < \len{u}$. Let $s$ be the maximal common suffix of $v^{\infty}$ and $v^{\infty}u$. If $\len{s} \geq \len{vu}$ then $u$ and $v$ are powers of the same word.
\end{lemma}
\begin{proof}
If $\len{s} \geq \len{vu}$ then $s$ has $vu$ as a suffix.  Since $v$ is a root of $s$, $v$ is a root of $u$  so $u = u^{\prime}v^{t}$ for some $t \geq 1$ and suffix $u^{\prime}$ of $v$.  Then $s$ has $u^{\prime}v^{t}v$ as a suffix since that is a suffix of $v^{\infty}$ and $\len{s} \geq \len{u^{\prime}v^{t}v}$.  Then $uv$ is a suffix of $s$ so $uv = vu$ as they are both suffixes of $s$ and have the same length so Lemma \ref{B1} gives the claim.
\end{proof}

\begin{lemma}\label{s2}
Let $v$ and $u$ be words with $\len{v} < \len{u}$ which are not powers of the same word and where $v$ is a suffix of $u$.  Let $s$ be the maximal common suffix of $v^{\infty}$ and $v^{\infty}u$ (which must be finite by Lemma~\ref{s}).  Then $s$ is a suffix of any left-infinite concatenation of $u$ and $v$.
\end{lemma}
\begin{proof}
By Lemma \ref{s}, $\len{s} < \len{vu}$ so we need only verify that $s$ is a suffix of $uv^{q}$ for $q \geq 1$ and of $uu$.  Since $v$ is a suffix of $u$, $uu$ has $vu$ as a suffix hence has $s$ as a suffix.  If $\len{s} \geq \len{u}$ then $v$ is a root of $u$ so $u = u^{\prime}v^{t}$ and $uv^{q} = u^{\prime}v^{t}v^{q}$ is a suffix of $v^{\infty}$ so $s$ is a suffix of $uv^{q}$.  If $\len{s} < \len{u}$ then $u = u_{0}s^{\prime}v^{t}$ for some (possibly empty) suffix $s^{\prime}$ of $v$ and $t \geq 1$ (as $s = s^{\prime}v^{t}$ has $v$ as a root and $\len{s} \geq \len{v}$ as $v$ is a suffix of $u$).  Then $uv^{q} = u_{0}s^{\prime}v^{t+q}$ has $s = s^{\prime}v^{t}$ as a suffix.
\end{proof}

\begin{lemma}\label{s3}
Let $v$ and $u$ be words and $s$ be the maximal common suffix of $v^{\infty}$ and $v^{\infty}u$.  Let $y$ and $z$ be suffixes of some (possibly distinct) concatenations of $u$ and $v$, both of length at least $\len{s}$.  Then for any word $w$, the maximal common suffix of $yvw$ and $zuw$ is $sw$.
\end{lemma}
\begin{proof}
Since $y$ is a suffix of a concatenation of $u$ and $v$, so is $yv$.  Then $yv$ has $sv$ as a suffix by Lemma \ref{s2}.  Likewise $zu$ has $su$ as a suffix.  As $s$ is a suffix of $v^{\infty}$, then so is $yv$.  Likewise, $zu$ is a suffix of $v^{\infty}u$.  Therefore the maximal common suffix of $yv$ and $zu$ is $s$ (as they are both at least as long as $s$).
\end{proof}

\begin{lemma}\label{6}
If $p(q + 1) - p(q) = 1$ then there exists a bi-special word which has length in $[q, q+p(q)]$, has exactly two successors, and is the unique right-special word of its length and also the unique left-special word of its length.
\end{lemma}
\begin{proof}
Let $w$ be the unique right-special word of length $q$ (which must have exactly two successors) and $y$ be the unique left-special word and write $z$ for the label of the path from $y$ to $w$ in the Rauzy graph.  Then $\len{z} \leq p(\len{w})$.  The word $yz$ is left-special and right-special and $\len{yz} = \len{y} + \len{z} \leq q + p(q)$.

If $x$ is a word of the same length as $yz$ which is right-special then $x$ must have $w$ as a suffix.  Then $x = x_{0}w$ and $\len{x_{0}} = \len{z}$.  Since there is only one path in the Rauzy graph ending at $w$ of length $\len{z}$ (due to $y$ being the unique left-special word), we have that $x = yz$.
\end{proof}

\begin{lemma}\label{ab}
Let $X$ be an infinite transitive subshift with $p(q) \leq \frac{4}{3}q$ for all sufficiently large $q$.  Then there exist words $a$ and $b$ which begin with different letters with $\len{a} < \len{b} < 2\len{a}$ and $p(q) < \frac{4}{3}q$ for all $q \geq \len{a}$ and where $a$ is a root of $b$ such that every $x \in X$ can be written in exactly one way as a concatenation of $a$ and $b$. If we define $s$ to be the maximal common suffix of $a^{\infty}$ and $a^{\infty}b$, there exists $t \geq 0$ so $sa^{t}$ is the unique right-special and left-special word of its length.
\end{lemma}
\begin{proof}
There exist infinitely many $q$ such that $p(q+1) - p(q) = 1$ by Corollary \ref{RScor}.  By Lemma \ref{6}, there exists a bi-special word $w$ with $\len{w}$ arbitrarily large which is the unique left-special and right-special word of its length and which has exactly two successors.  We may assume $p(q) \leq \frac{4}{3}q$ for all $q \geq \len{w}$.
We note that by \cite{ormespavlov}, $X$ is infinite and minimal.

Let $u$ and $v$ be the shortest two return words for $w$ (meaning $wu$ and $wv$ both have $w$ as a suffix) which will be the labels of the two paths from $w$ to itself in the Rauzy graph $G_{X,\len{w}}$ for words of length $\len{w}$, with $v$ being the shorter of the two.  All bi-infinite words in $X$ can be written in exactly one way as a concatenation of $v$ and $u$, as every such word must be the label of a path in the Rauzy graph (which visits the vertex $w$ infinitely many times by minimality of $X$), and the only two such paths have labels $v$ and $u$.

Since $\len{u} + \len{v} \leq p(\len{w}) + 1 \leq \frac{4}{3}\len{w} + 1$, we have $2\len{v} \leq \frac{4}{3}\len{w} + 1$ so $\len{v} \leq \frac{2}{3}\len{w} + \frac{1}{2}$. This is less than $\len{w}$ (since $\len{w} > 1$), and so $v$ is a root of $w$ by Lemma \ref{B}.  Note that $v$ cannot be a proper power of any word since if $v = v_{0}^{t}$ then $wv_{0}$ has $w$ as a suffix so $v_{0}$ is a root of $w$ making $v_{0}$ a return word for $w$ which is shorter than $v$.

Observe that if $\len{w} < 3 \len{v}$ then $\len{u} \leq \frac{4}{3} \len{w} + 1 - \len{v} < \frac{4}{3}\len{w} - \frac{1}{3} \len{w} + 1$ so $u$ is a suffix of $w$ making $v$ a root of $u$. We write $u = u^{\star}v^{s}$ for some proper suffix $u^{\star}$ of $v$ (which cannot be empty as $u$ and $v$ start with different letters) and define $a = v$ and $b = u^{\star}v$. Then as before, every bi-infinite word in $X$ can be written uniquely as a concatenation of $v = a$ and $u = ba^{s-1}$, hence the same is true of $a$ and $b$ (since $a = v$).  Clearly $a$ is a root of $b$, and $\len{a} < \len{b} < 2\len{a}$ as $0 < \len{u^{\star}} < \len{a}$.

So assume from here on that $\len{w} \geq 3 \len{v}$.

Suppose now that for every suffix $w_{0}$ of $w$ with $\len{v} \leq \len{w_{0}} < 2 \len{v}$, we have $p(\len{w_{0}}+1) - p(\len{w_{0}}) \geq 2$.  Then, by Corollary \ref{RScor}, $p(2\len{v}) = p(2 \len{v}) - p(\len{v}) + p(\len{v}) \geq 2(2 \len{v} - \len{v}) + \len{v} + 1 = 3\len{v} + 1$ so $\frac{p(2 \len{v})}{2\len{v}} > \frac{3}{2}$, contradicting our hypothesis.

Therefore there exists $w_{0}$ a suffix of $w$ with $\len{v} \leq \len{w_{0}} < 2 \len{v}$ which is the unique right-special word of its length and it has exactly two successors.

Since $w_{0}$ is a suffix of $w$, $v$ is a root of $w_{0}$.
As there must also be a unique left-special word of the same length as $w_{0}$, $w_{0}$ extends to a bi-special word $w_{00}$ which is the unique left-special and right-special word of its length and which has exactly two successors (Lemma \ref{6}).  Now $\len{w_{00}} \leq \len{w_{0}} + \len{v}$ since the path from the left-special to the right-special vertex in the Rauzy graph for words of length $\len{w_{0}}$ must be no longer than $v$ (as $w_{0}v$ must have $w_{0}$ as a suffix).  Then $\len{w_{00}} < 2 \len{v} + \len{v} = 3\len{v} \leq \len{w}$ so $w_{00}$ is a proper suffix, and prefix, of $w$.

Let $v_{0}$ and $u_{0}$ be the shortest return words for $w_{00}$ with $v_{0}$ beginning with the same letter as $v$ (and $u_0$ beginning with a different letter). Then all bi-infinite words in $X$ are concatenations of $u_{0}$ and $v_{0}$.  Since $v$ is a return word for $w_{00}$, $v$ must be a concatenation of $u_{0}$ and $v_{0}$ which means that $v_{0}$ must be a prefix of $v$ by virtue of sharing a common first letter. Likewise $u_{0}$ must be a prefix of $u$.

Since $v$ is a suffix of $w$, then $vv_{0}$ has $v$ as a suffix so $v_{0}$ is a root of $v$ by Lemma \ref{B}.  Write $v = v^{\prime}v_{0}^{t}$ for some $t \geq 1$ and $v^{\prime}$ a proper suffix of $v_{0}$.  Then $v_{0} = v^{\prime\prime}v^{\prime}$ so $v$ has $v^{\prime}v_{0} = v^{\prime}v^{\prime\prime}v^{\prime}$ as a prefix.  But $v_{0}$ is also a prefix of $v$ so both $v^{\prime}v^{\prime\prime}$ and $v^{\prime\prime}v^{\prime}$ are prefixes of $v$.  Therefore they are equal so by Lemma \ref{B1} both are powers of the same word.  But then $v$ is a power of that word and it cannot be a proper power of any word so either $v^{\prime}$ or $v^{\prime\prime}$ is empty and so $v_{0} = v$.

If $\len{u_{0}} \leq \len{v}$ then $u_{0}$ is a root of $w_{00}$ hence of $v$.  Write $v = v^{\star}u_{0}^{s}$ for some proper suffix $v^{\star}$ of $u_{0}$ (which cannot be empty as $v$ begins with a different letter than $u$) and $s \geq 1$.
Taking $a = u_{0}$ and $b = v^{\star}u_{0}$, then every bi-infinite word in $X$ is a concatenation of $u_{0} = a$ and $v = ba^{s-1}$.  Clearly $a$ is a root of $b$ and $\len{a} < \len{b} < 2\len{a}$.

So we are left with $\len{u_{0}} > \len{v}$.  Here $\len{u_{0}} \leq p(\len{w_{00}}) + 1 - \len{v}
< \frac{4}{3}\len{w_{00}} + 1 - \frac{1}{3}\len{w_{00}}$ as $\len{w_{00}} < 3\len{v}$.  Therefore $\len{u_{0}} \leq \len{w_{00}}$.  So $u_{0}$ is a suffix of $w$ hence $v$ is a root of $u_{0}$.  Writing $u_{0} = u^{\star}v^{s}$ for some proper suffix $u^{\star}$ of $v$ and $s \geq 1$ then taking $a = v$ and $b = u^{\star}v$, just as before we have that every bi-infinite word in $X$ is a unique concatenation of $v = a$ and $u_{0} = ba^{s-1}$, hence of $a$ and $b$. As before, clearly $a$ is a root of $b$ and $\len{a} < \len{b} < 2\len{a}$.

In all cases, one of $a,b$ is a prefix of $u$ and the other is a prefix of $v$.  Since $u$ and $v$ begin with different letters, $a$ and $b$ begin with different letters. It remains to verify the claim about the maximal common suffix $s$ and that $a$ may be taken arbitrarily long.

In the case when $a$ is a root of $w$ (and $w_{00}$ was not introduced), set $w_{00} = w$ and $t=0$.
Then in all cases, $a$ is a root of $w_{00}$ as $a$ is either $v$ or $u_{0}$ so $w_{00}$ is a suffix of $a^{\infty}$.  In all cases, $ba^{t}$ is the other return word for $w_{00}$ for some $t \geq 0$.  Then $w_{00}a^{\ell}ba^{t}$ has $w_{00}$ as a suffix for all $\ell \geq 0$ so $w_{00}$ is a suffix of $a^{\infty}ba^{t}$.  Since $w_{00}$ is left-special and $a$ and $ba^{t}$ are its two return words, the maximal common suffix of $a^{\infty}$ and $a^{\infty}ba^{t}$ must be no longer than $w_{00}$.  Therefore $w_{00} = sa^{t}$ where $s$ is the maximal common suffix of $a^{\infty}$ and $a^{\infty}ba$.

Let $\{ w_{\ell} \}$ be a sequence of such bi-special words with $\len{w_{\ell}}$ increasing to $\infty$ and let $\{ a_{\ell} \}$ and $\{ v_{\ell} \}$ be the corresponding $a$ and $v$ above.  Since either $a = v$ or $a = u_{0}$, and in both cases it is a root of $w_{00}$, $a_{\ell}$ is a root of $v_{\ell}$.

Since $w_{\ell}$ is the unique right-special word of its length, it is a suffix of $w_{\ell+1}$ and therefore $v_{\ell}$ is a suffix of $v_{\ell+1}$.  If $\len{v_{\ell}}$ were bounded then there would exist $L$ such that $v_{\ell} = v_{L}$ for $\ell \geq L$ but then $v_{L}$ would be a root of $w_{\ell}$ for $\ell \geq L$ so $v_{L}^{\infty} \in X$, a contradiction.  So $\len{v_{\ell}} \to \infty$.  Likewise, since $a_{\ell}$ is a root of $v_{\ell}$, if $\len{a_{\ell}}$ were bounded then for some $L$ we would have $a_{L}^{\infty} \in X$.  Therefore $\len{a_{\ell}} \to \infty$ so we may take $a$ and $b$ such that for all $q \geq \len{a}$, we have $p(q) < \frac{4}{3}q$.
\end{proof}

The following lemma is our main tool to recursively demonstrate the structure from Proposition~\ref{decomp}. The key is control over the lengths of the suffixes from Lemmas~\ref{s} and \ref{s2}.

\begin{lemma}\label{induction}
Let $X$ be an infinite transitive subshift with $\frac{p(q)}{q} < \frac{4}{3}$ for $q > N$.  Let $u$ and $v$ be words with $N < \len{v} < \len{u}$ such that $v$ is a suffix of $u$ and $v$ is not a prefix of $u$.  Let $s$ be the maximal common suffix of $v^{\infty}$ and $v^{\infty}u$ and let $p$ be the maximal common prefix of $u$ and $v$.

Assume that $\len{p} + \len{s} < \len{u} + \len{v}$ and $\len{p} + \len{s} < 3\len{v}$ and that every bi-infinite word in $X$ can be written as a concatenation of $u$ and $v$.  Then there exist $0 < m < n$ such that every concatenation of $u$ and $v$ which represents a point in $X$ has only $v^{m-1}$ and $v^{n-1}$ appearing between nearest occurrences of $u$ and satisfying:
\begin{itemize}
\item $n \leq 2m$ whenever $m > 1$;
\item $n < 1.9m$ whenever $m > 4$;
\item $n \leq 3$ whenever $m = 1$
\end{itemize}
and the words $sv^{n-2}p$ and $sv^{m-1}uv^{m-1}p$ are right-special.
\end{lemma}
\begin{proof}
For brevity, whenever we refer to a `concatenation' in the following, it is a concatenation of $u,v$ which represents a point of $X$ or a subword of such a point. We again note that by \cite{ormespavlov}, $X$ is infinite and minimal, and so no concatenation can contain infinitely many consecutive $v$. Similarly, if there was only a single number of $v$ which may occur between nearest occurrences of $u$, then $X$ would be finite, contradicting our assumptions. So there are at least two different numbers of $v$ which can occur between nearest occurrences of $u$.

Suppose for a contradiction that $uv^{x}u$ and $uv^{y}u$ and $uv^{z}u$ all appear in some concatenations and that $x < y < z$. We may assume that $x$ is the minimal value such that $uv^{x}u$ appears in a concatenation.
Since $uv^{x}u$ and $uv^{y}u$ are necessarily preceded by $v^{x}$ (due to $x$ being minimal), then $v^{x}uv^{x}u$ and $v^{x}uv^{x}v$ both appear in concatenations (as $y > x$).  By Lemma \ref{s2} (as $v$ is not a prefix of $u$, they cannot be powers of the same word), $s$ is a suffix of every left-infinite concatenation. This means that $v^x u v^x u$ and $v^x u v^x v$ are both preceded by $s$ in the bi-infinite concatenations they respectively appear in, and so $sv^{x}uv^{x}$ can be followed by either $u$ or $v$, meaning that $sv^{x}uv^{x}p$ is right-special (since the letters appearing after $p$ in $u$ and $v$ are distinct by maximality of $p$).

Likewise, $v^{x}uv^{y}u$ and $v^{x}uv^{y}v$ appear in some concatenations (due to $z > y$) so $sv^{x}uv^{y}p$ is also right-special.  By Lemma \ref{s3}, the maximal common suffix of $sv^{x}uv^{x}p$ and $sv^{x}uv^{y}p$ is $sv^{x}p$. Therefore there are at least two right-special words of length $\ell$ for $\len{sv^{x}p} < \ell \leq \len{sv^{x}uv^{x}p}$ (namely, the unequal suffixes of $sv^{x}uv^{x}p$ and $sv^{x}uv^{y}p$ of length $\ell$).  Then, since $\len{p} + \len{s} < \len{v} + \len{u} < 2\len{u}$, by Corollary~\ref{RScor}
\begin{align*}
\frac{p(\len{sv^{x}uv^{x}p})}{\len{sv^{x}uv^{x}p}} &\geq 1 + \frac{\len{sv^{x}uv^{x}p} - \len{sv^{x}p}}{\len{sv^{x}uv^{x}p}}
= 1 + \frac{x\len{v} + \len{u}}{\len{p} + \len{s} + 2x\len{v} + \len{u}} 
> 1 + \frac{x\len{v} + \len{u}}{2\len{u} + 2x\len{v} + \len{u}}.
\end{align*}
The final expression is increasing for $x \geq 0$, hence is at least $\frac{4}{3}$ (its value at $x = 0$), contradicting our hypothesis that $p(q)/q < \frac{4}{3}$ for $q > N$. Therefore such $x < y < z$ cannot exist so there are only two distinct values $x$ and $y$.  Writing $x = m - 1$ and $y = n - 1$ then shows that $v^{m-1}$ and $v^{n-1}$ are the only words appearing between occurrences of $u$ in a concatenation.

By similar reasoning as above, we observe that $sv^{m-1}uv^{m-1}p$ is right-special and that $sv^{n-2}p$ is also right-special since $sv^{n-1}u$ appears in a concatenation and it has $sv^{n-2}v$ as a prefix and $sv^{n-2}u$ as a suffix. Again by similar reasoning as above, their maximal common suffix is $sv^{m-1}p$.

Suppose $\len{sv^{m-1}uv^{m-1}p} \leq \len{sv^{n-2}p}$.  Then there are at least two right-special words of length $\ell$ for $\len{sv^{m-1}p} < \ell \leq \len{sv^{m-1}uv^{m-1}p}$ so, by Corollary~\ref{RScor} and the fact that $\len{p} + \len{s} < \len{u} + \len{v} < 2\len{u}$,
\[
\frac{p(\len{sv^{m-1}uv^{m-1}p})}{\len{sv^{m-1}uv^{m-1}p}}
\geq 1 + \frac{(m-1)\len{v} + \len{u}}{\len{p} + \len{s} + 2(m-1)\len{v} + \len{u}}
> 1 + \frac{(m-1)\len{v} + \len{u}}{2(m-1)\len{v} + 3\len{u}} \geq \frac{4}{3}
\]
which contradicts our hypothesis.

So instead $\len{sv^{n-2}p} < \len{sv^{m-1}uv^{m-1}p}$.  Then there are at least two right-special words of length $\ell$ for $\len{sv^{m-1}p} < \ell \leq \len{sv^{n-2}p}$ so, by Corollary~\ref{RScor} and the fact that $\len{p} + \len{s} < 3\len{v}$,
\[
\frac{p(\len{sv^{n-2}p})}{\len{sv^{n-2}p}} \geq 1 + \frac{(n - m - 1)\len{v}}{\len{p} + \len{s} + (n-2)\len{v}}
> 1 + \frac{(n-m-1)\len{v}}{3\len{v} + (n-2)\len{v}}
= 1 + \frac{n-m-1}{n+1}.
\]

Consider first when $m=1$.  If $n \geq 4$ then $\frac{n-m-1}{n+1} = \frac{n-2}{n+1} \geq \frac{2}{5} > \frac{1}{3}$ which contradicts our hypothesis.  

Now consider when $m > 1$.  If $n \geq 2m + 1$ then $\frac{n-m-1}{n+1} \geq \frac{2m+1 - m - 1}{2m+1+1} = \frac{m}{2m+2} \geq \frac{2}{2(2)+2} = \frac{1}{3}$ contradicting our hypothesis.  So $n \leq 2m$ when $m > 1$.

Finally, consider when $m \geq 5$.  Suppose $n \geq 1.9m$.  Then
\[
\frac{n-m-1}{n+1} \geq \frac{1.9m - m - 1}{1.9m + 1} = \frac{0.9m - 1}{1.9m + 1} \geq \frac{4.5 - 1}{9.5 + 1} = \frac{1}{3}
\]
contradicting our hypothesis.  So $n < 1.9m$ whenever $m > 4$.
\end{proof}

\begin{proof}[Proof of Proposition~\ref{decomp}]
We prove by induction that such sequences exist, using the notation $v_k := \pi(\rho_{k-1}(0))$ and $u_k := \pi(\rho_{k-1}(1))$. 

By \cite{ormespavlov}, $X$ is minimal.  Write $s_k$ for the maximal common suffix of $v_{k}^{\infty}$ and $v_{k}^{\infty}u_{k}$ and $p_{k}$ for the maximal common prefix of $v_{k}$ and $u_{k}$.

Our inductive hypotheses are the following:
\begin{itemize}
\item all $x \in X$ can be written as concatenations of $u_k$ and $v_k$;
\item $v_k$ is a suffix of $u_k$ and is not a prefix of $u_k$;
\item $\len{p_k} + \len{s_k} < \min(\len{v_k} + \len{u_k}, 3\len{v_k})$;
\item $v_{k} = (\pi \circ \tau_{m_{1},n_{1}} \circ \cdots \circ \tau_{m_{k-1},n_{k-1}})(0)  = \pi(\rho_{k-1}(0))$ and $u_{k} = (\pi \circ \tau_{m_{1},n_{1}} \circ \cdots \circ \tau_{m_{k-1},n_{k-1}})(1) = \pi(\rho_{k-1}(1))$.
\end{itemize}

Since $\limsup \frac{p(q)}{q} < \frac{4}{3}$, eventually $p(q) < \frac{4}{3}q$.  Lemma \ref{ab} gives $v_{1}$ and $u_{1}$ with 
$v_{1}$ a suffix of $u_{1}$ and $\len{v_1} < \len{u_{1}} < 2\len{v_{1}}$ which start with different letters such that every infinite word is a concatenation of $u_{1}$ and $v_{1}$.  By Lemma \ref{s}, $\len{s_{1}} < \len{v_{1}u_{1}} < 3\len{v_{1}}$.  As $u_{1}$ and $v_{1}$ begin with different letters, $p_1$ is empty.  Therefore the base case is established by setting $\pi(0) = v_{1}$ and $\pi(1) = u_{1}$.  Lemma \ref{ab} ensures that $p(q) < \frac{4}{3}q$ for all $q \geq \len{\pi(0)}$.

Given $v_k$ and $u_k$, by Lemma \ref{induction} there exist $0 < m_k < n_k$ such that every infinite word is a concatenation of $v_{k+1} = v_{k}^{m_{k}-1}u_{k}$ and $u_{k+1} = v_{k}^{n_{k}-1}u_{k}$.  
Observe that $u_{k+1} = v_{k}^{n_{k}-1}u_{k} = 
(\pi(\rho_{k-1}(0)))^{n_{k}-1} \pi(\rho_{k-1}(1)) = 
\pi(\rho_{k-1}(0^{n_k - 1} 1)) = 
\pi(\rho_{k-1}(\tau_{m_k,n_k}(1))) = \pi(\rho_k(1))$ and similarly 
$v_{k+1} = \pi(\rho_k(0))$.

Clearly $v_{k+1}$ is a suffix of $u_{k+1}$.  If $v_{k+1}$ were a prefix of $u_{k+1}$ then $u_{k}$ would be a prefix of $v_{k}^{n_{k}-m_{k}}u_{k}$ but that would make $v_{k}$ a prefix of $u_{k}$.  So $v_{k+1}$ is not a prefix of $u_{k+1}$, and $p_{k+1} = v_{k}^{m_{k}-1}p_{k}$.

By definition, $s_{k+1}$ is the maximal common suffix of $v_{k+1}^{\infty}$ and $v_{k+1}^{\infty} u_{k+1}$. We can rewrite these as $y = \ldots u_k v_k^{m_k - 1} u_k$ and $z = \ldots v_k v_k^{m_k - 1} u_k$. These share a suffix of $v_k^{m_k - 1} u_k$, so we must just find the maximal common suffix of the portions with this removed, i.e. $y' = \ldots u_k$, a concatenation ending with $u_k$, and $z' = \ldots v_k$, a concatenation ending with $v_k$. But $y'$ then agrees with $v_k^{\infty} u_k$ on a suffix of length $|u_k| + |s_k| > |s_k|$ by Lemma~\ref{s2} and $z'$ agrees with $v_k^{\infty}$ on a suffix of length $|v_k| + |s_k| > |s_k|$ by Lemma~\ref{s2}, meaning that $y'$ and $z'$ have maximal common suffix $s_k$. Therefore, $s_{k+1} = s_{k}v_{k}^{m_{k}-1}u_{k} = s_{k}v_{k+1}$. Then,
\[
\len{p_{k+1}} + \len{s_{k+1}} = \len{p_k} + \len{s_k} + 2(m_{k}-1)\len{v_k} + \len{u_k}
< (2m_k - 1)\len{v_k} + 2\len{u_k}
= 2\len{v_{k+1}} + \len{v_k}
\]
and since $\len{v_{k+1}} + \len{v_k} \leq \len{u_{k+1}}$ and $\len{v_k} < \len{v_{k+1}}$, the inductive hypotheses are verified.

Lemma \ref{induction} gives that $n_k \leq 2m_k$ when $m_k > 1$ and $n_k \leq 1.9m_k$ when $m_k > 4$ and that $n_k \leq 3$ when $m_k = 1$.

Suppose that $m_{k} = 1$ and $n_{k} = 3$ and  $n_{k-1} \geq m_{k-1} + 2$.  By Lemma \ref{induction}, the words $s_{k}v_{k}p_{k}$ and $s_{k-1}v_{k-1}^{n_{k-1}-2}p_{k}$ and $s_{k}u_{k}p_{k}$ are right-special.  By Lemma \ref{s3}, the maximal common suffix of $s_{k}v_{k}p_{k}$ and $s_{k}u_{k}p_{k}$ is $s_{k}p_{k}$.  Using Lemma \ref{s2} and that $p_{k} = v_{k-1}^{m_{k-1}-1}p_{k-1}$, both $s_{k}v_{k}p_{k}$ and $s_{k}u_{k}p_{k}$ have $s_{k-1}u_{k-1}v_{k-1}^{m_{k-1}-1}p_{k-1}$ as a suffix.  By Lemma \ref{s3}, the maximal common suffix of either of them and $s_{k-1}v_{k-1}^{n_{k-1}-2}p_{k-1}$ is then $s_{k-1}v_{k-1}^{m_{k-1}-1}p_{k-1}$.  Therefore there are least
$
\len{s_{k}v_{k}p_{k}} + \len{s_{k}v_{k}p_{k}} - \len{s_{k}p_{k}} + \len{s_{k-1}v_{k-1}^{n_{k-1}-1}p_{k-1}} - \len{s_{k-1}v_{k-1}^{m_{k-1}-1}p_{k-1}} 
$
right-special words of length at most $\len{s_{k}v_{k}p_{k}}$.

Since $p_{k} = v_{k-1}^{m_{k}-1}p_{k-1}$, $s_{k} = s_{k-1}v_{k}$ and $\len{p_{k-1}} + \len{s_{k-1}} < 3\len{v_{k-1}}$,
\[
\len{p_{k}} + \len{s_{k}} = (m_{k-1}-1)\len{v_{k-1}} + \len{v_{k}} + \len{p_{k-1}} + \len{s_{k-1}} < \len{v_{k}} + (m_{k-1} + 2)\len{v_{k-1}} = 2\len{v_{k}} - \len{u_{k}} + 3\len{v_{k-1}}.
\]
Therefore, since $n_{k-1} \geq m_{k-1} + 2$,
\begin{align*}
\frac{p(\len{s_{k}v_{k}^{n_{k}-2}p_{k}})}{\len{s_{k}v_{k}^{n_{k}-2}p_{k}}}
&\geq 1 + \frac{\len{v_{k}} + (n_{k-1} - m_{k-1} - 1)\len{v_{k-1}}}{\len{v_{k}} + \len{p_{k}} + \len{s_{k}}} 
> 1 + \frac{\len{v_{k}} + \len{v_{k-1}}}{3\len{v_{k}} - \len{u_{k-1}} + 3\len{v_{k-1}}}
> 1 + \frac{1}{3}
\end{align*}
contradicting our hypothesis.  So if $m_{k}=1$ and $n_{k} = 3$ then $n_{k-1} = m_{k-1} + 1$.

Since $s_{1}v_{1}^{t}$ is the unique right-special and unique left-special word of its length for some $t \geq 0$ (Lemma \ref{ab}) and $u_{1}v_{1}^{t}$ and $v_{1}$ are the two return words for $s_{1}v_{1}^{t}$, we have that $t \leq m_{1} - 1$ as $u_{1}$ is always followed by $v_{1}^{t}$.  Since $s_{1}v_{1}^{t}$ is left-special, $u_{1}v^{t}s_{1}v^{t}$ must appear meaning that $t = m_{1} - 1$.  Therefore any right-special word of length at least $\len{s_{1}v_{1}^{m_{1}-1}}$ must have $s_{1}v_{1}^{m_{1}-1}$ as a suffix.  As the return words for $s_{1}v_{1}$ are $v_{1}$ and $u_{1}v_{1}^{m_{1}-1}$, then every right-special word of at least that length is a suffix of a concatenation of $u_{1}$ and $v_{1}$.

Finally, since $v_k$ is in the language for all $k$, there exists a two-sided sequence containing $x^{(m_k), (n_k)} = \lim v_{k}$. Then since $X$ is minimal, $X$ is the orbit closure of $x^{(m_k), (n_k)}$.
\end{proof}

\begin{remark}\label{ps}
In future arguments, for any subshift $X$ satisfying the structure of Proposition~\ref{decomp}, we use the notation of the proof, i.e. $u_k = \pi(\rho_{k-1}(1))$, $v_k = \pi(\rho_{k-1}(0))$, $p_k$ is the maximal prefix of $v_k$ and $u_k$, and $s_k$ is the maximal suffix of $v_k^{\infty}$ and $v_k^{\infty} u_k$. In addition, as shown in the proof of Proposition~\ref{decomp}, the sequence $(p_k)$ satisfies the recursion $p_{k+1} = v_{k}^{m_{k}-1}p_{k} = v_k p_{k+1}$, the sequence $(s_k)$ satisfies the recursion $s_{k+1} = s_{k}v_{k+1}$, and $|p_k| + |s_k| < \min(|u_k| + |v_k|, 3|v_k|)$ for all $k$.
\end{remark}

\begin{remark}\label{ud}
By induction on $k$, each substitution $\pi \circ \rho_k$ is uniquely decomposable, in the sense that each $x \in X$ can be decomposed uniquely into words $(\pi \circ \rho_k)(a)$ for $a \in \{0,1\}$. For $k = 0$, this follows from Lemma 2.9 since
$\pi(0) = v_1$ and $\pi(1) = u_1$ were constructed using that lemma. If $\pi \circ \rho_k$ is uniquely decomposable, then every $x$ is representable uniquely as a concatenation of $(\pi \circ \rho_k)(0)$ and $(\pi \circ \rho_k)(1)$, and then the same must be true
of $(\pi \circ \rho_{k+1})(0) = (\pi \circ \rho_k)(0)^{m_{k+1} - 1} (\pi \circ \rho_k)(1)$ and 
$(\pi \circ \rho_{k+1})(1) = (\pi \circ \rho_k)(0)^{n_{k+1} - 1} (\pi \circ \rho_k)(1)$ (since each of these contains 
$(\pi \circ \rho_k)(1)$ exactly once.)
\end{remark}

\section{Subshifts with \texorpdfstring{$C < \nicefrac{4}{3}$}{C < 4/3} have discrete spectrum}

\begin{theorem}\label{disc}
If $X$ is an infinite transitive subshift with $\limsup \frac{p(q)}{q} < \frac{4}{3}$, then $X$ is uniquely ergodic with unique measure which has discrete spectrum.
\end{theorem}

Our proof relies on first proving exponential decay of some quantities, which will later be used to verify discrete spectrum via so-called mean almost periodicity.





\begin{proposition}\label{error}
Let $X$ be the orbit closure of $x^{(m_{k}),(n_{k})}$ where $(m_k), (n_k)$ satisfy the conclusions of Proposition~\ref{decomp}.
Then there exist $\epsilon_k$ which converge to $0$ exponentially so that for every $k$, 
\[
\frac{(n_{k+1}+1)\len{\pi(0)}\prod_{i=1}^k (n_i - m_i)}{\len{(\pi \circ \rho_{k+1})(0)}}
< \epsilon_k.
\]
\end{proposition}

\begin{proof}
We first set some preliminary notation. Define $a_{1} = 1$ and $a_{k} = n_{k-1} - m_{k-1}$ and $b_{k} = m_{k}$ for $k > 0$. Note that by Proposition~\ref{decomp}, all $b_k$ and $a_k$ are positive; $a_{k+1} \leq b_k$ whenever $b_k > 1$; $a_{k+1} < 0.9b_{k}$ whenever $b_{k} > 4$; and $a_{k+1} \leq 2$ whenever $b_{k} = 1$. We also define $d_{k} = \len{(\pi \circ \rho_k)(0)}$, and note that $(d_k)$ satisfies the recursion
\begin{equation}\label{rec}
d_{k+1} = b_{k+1} d_k + a_{k+1} d_{k-1}
\end{equation}
where $d_{-1} = \len{\pi(1)} - \len{\pi(0)}$ and $d_{0} = \len{\pi(0)}$.



%

For ease of notation, define 
\[
\beta_j = \frac{a_{j+1} d_{j-1}}{d_j} 
\]
and observe that, by (\ref{rec}),
\[
\beta_{j+1} = \frac{a_{j+2}d_{j}}{d_{j+1}} = \frac{a_{j+2}}{b_{j+1} + a_{j+1}\frac{d_{j-1}}{d_{j}}} = \frac{a_{j+2}}{b_{j+1} + \beta_{j}}. 
\]

Note that $\beta_{0} = \frac{a_{1}d_{-1}}{d_{0}} = \frac{d_{-1}}{\len{\pi(0)}}$.  Then
\begin{equation}\label{betaprod}
\frac{\len{\pi(0)}a_{1}\cdots a_{k+1}}{d_{k}} = \frac{\len{\pi(0)}}{d_{-1}}\prod_{j=0}^{k} \frac{a_{j+1}d_{j-1}}{d_{j}}
= \frac{\len{\pi(0)}}{d_{-1}}\beta_{0} \prod_{j=1}^{k} \beta_{j} = \prod_{j=1}^{k} \beta_{j}.
\end{equation}

\begin{claim}
$0 < \beta_{j} < 2$ for all $j \geq 0$.
\end{claim}
\begin{proof}
Since $a_{j+1} \leq b_{j} + 1$ for all $j$, $\beta_{j} \leq \frac{b_{j}+1}{b_{j}+\beta} < 1 + \frac{1}{b_{j}} \leq 2$.
\end{proof}

\begin{claim}
If $a_{j+1} \leq b_{j}$ then $\beta_{j} < 1$.
\end{claim}
\begin{proof}
Since $\beta_{j-1} > 0$,
$\beta_{j} = \frac{a_{j+1}}{b_{j} + \beta_{j-1}} < \frac{a_{j+1}}{b_{j}} \leq 1$.
\end{proof}

\begin{claim}
If $a_{j+1} = b_{j} + 1$ then at least one of $\beta_{j} < 1$ or $\beta_{j}\beta_{j-1} \leq 1$.
\end{claim}
\begin{proof}
When $a_{j+1} = 2$ and $b_{j} = 1$, by Proposition \ref{decomp}, $\tau_{1,3}$ cannot occur for consecutive values so we have $a_{j} \leq b_{j}$ so $\beta_{j-1} \leq 1$.  Since $\beta_{j} = \frac{2}{1 + \beta_{j-1}} \geq 1$, we have $\beta_{j}\beta_{j-1} = 2 - \beta_{j} \leq 1$.  
\end{proof}

This implies $\prod_{j=1}^{k} \beta_{j} \leq \beta_{1} \leq 2$.

%
%

By the assumptions on $(m_k)$ and $(n_k)$, we see that $a_{k+1} \leq b_k$ when $b_k > 1$ and $a_{k+1} \leq 2$ when $b_{k}=1$ and 
$a_{k+1} < 0.9 b_k$ when $b_k > 4$. We now break into several cases.

\noindent
\textbf{Case 1:} If $b_{j} > 4$ then $\beta_{j} < 0.9$.

\begin{proof}
If $b_{j} > 4$ then, as $d_{j} > b_{j}d_{j-1}$ by (\ref{rec}), $\beta_{j} = \frac{a_{j+1}d_{j-1}}{d_{j}} < 0.9$.
\end{proof}

\noindent
\textbf{Case 2:}
If $a_{j+1} \leq b_{j} \leq 4$ and $b_{j-1} \leq 4$ then $\beta_{j} < 0.96$.

\begin{proof}
If $a_{j+1} \leq b_{j} \leq 4$ and $b_{j-1} \leq 4$ then by (\ref{rec}),
\begin{equation*}
d_{j} = b_{j} d_{j-1} + a_{j} d_{j-2} \leq b_{j} d_{j-1} + (b_{j-1}+1) d_{j-2} <
b_{j} d_{j-1} + d_{j-1} + d_{j-2} \leq (b_{j} + 2)d_{j-1} \leq 6d_{j-1}.
\end{equation*}

Then, again by (\ref{rec}), using that $a_{j+1} \leq b_{j}$,
\[
\frac{d_{j}}{d_{j-1}} = b_{j} + \frac{a_{j} d_{j-2}}{d_{j-1}} > b_{j} + \frac{1}{6} \geq a_{j+1} + \frac{1}{6}.
\]

Therefore, since $a_{j+1} \leq b_{j} \leq 4$,
\[
\beta_{j} = \frac{a_{j+1} d_{j-1}}{d_{j}} 
< \frac{a_{j+1}}{a_{j+1} + (1/6)} =
\left(1 + \frac{1}{6a_{j+1}}\right)^{-1} <\left(1 + \frac{1}{24}\right)^{-1} = 0.96. \qedhere
\]
\end{proof}


\noindent
\textbf{Case 3:}
If $a_{j+1} \leq b_{j} \leq 4$ and $b_{j-1} > 4$ then at least one of $\beta_{j} < 0.96$ or $\beta_{j}\beta_{j-1} < 0.5$ holds.

\begin{proof}
Consider when $a_{j+1} \leq b_{j} \leq 4$ and $b_{j-1} > 4$ so $\beta_{j-1} < 0.96$.  Suppose $\beta_{j} > \frac{8}{9}$.  Then
\[
\frac{8}{9} < \frac{a_{j+1}}{b_{j} + \beta_{j-1}} \leq \frac{b_{j}}{b_{j} + \beta_{j-1}} \leq \frac{4}{4 + \beta_{j-1}}
\]
so $8 + 2\beta_{j-1} < 9$ so $\beta_{j-1} < \frac{1}{2}$.  Then $\beta_{j}\beta_{j-1} < \beta_{j-1} < 0.5$ since $a_{j+1} \leq b_{j}$ implies $\beta_{j} < 1$.  So at least one of $\beta_{j} \leq \frac{8}{9} < 0.96$ or $\beta_{j}\beta_{j-1} < 0.5$ must hold.
\end{proof}

Any $j$ where $a_{j+1} \leq b_j$ is covered by Case 1 if $b_j > 4$ and Case 2 or 3 if $b_j \leq 4$. The only remaining case is then 
$a_{j+1} > b_j$, which happens only if $a_{j+1} = 2$ and $b_j = 1$.

\noindent
\textbf{Case 4:}
If $a_{j+1} = 2$ and $b_{j} = 1$ then at least one of $\beta_{j}\beta_{j-1} < \frac{48}{49}$ or $\beta_{j}\beta_{j-1}\beta_{j-2} < 0.52$ holds.

\begin{proof}
Consider any such $j$. By Proposition \ref{decomp}, $\tau_{1,3}$ cannot occur consecutively so $a_j \leq b_{j-1}$, and so $j-1$ is in one of Cases 1-3. 
If $\beta_{j-1} < 0.96$, then
\[
\beta_{j}\beta_{j-1} = \frac{a_{j+1}}{b_{j} + \beta_{j-1}}\beta_{j-1} = \frac{2\beta_{j-1}}{1 + \beta_{j}-1} = 1 + \frac{\beta_{j-1} - 1}{\beta_{j-1} + 1} < 1 + \frac{0.96 - 1}{0.96 + 1} = \frac{48}{49}.
\]

If $\beta_{j-1} \geq 0.96$, then $j-1$ must be in Case 3 and $\beta_{j-1}\beta_{j-2} < 0.5$.  Then
\[
\beta_{j}\beta_{j-1}\beta_{j-2} = \frac{2}{1 + \beta_{j-1}}\beta_{j-1}\beta_{j-2} < \frac{1}{1 + \beta_{j-1}} \leq \frac{1}{1 + 0.96} < 0.52. \qedhere
\]
\end{proof}

\begin{claim}
For all $k \geq 1$,
\[
\prod_{j=1}^{k} \beta_{j} < 2 \Big{(}\frac{48}{49}\Big{)}^{k/2}.
\]
\end{claim}
\begin{proof}
All $j>2$ are in one of the cases above, and so at least one of the following hold:
$\beta_{j} < 0.96$, 
$\beta_{j}\beta_{j-1} < \frac{48}{49}$, or
$\beta_{j}\beta_{j-1}\beta_{j-2} < 0.52$. For every $k$, we can group the product $\prod_{j=1}^{k} \beta_{j}$ into products of one, two, or three consecutive terms bounded from above in this way, with the possible exception of $\beta_1$ or $\beta_1 \beta_2$. As $0.96 < \sqrt{\frac{48}{49}}$ and $0.52 < (\frac{48}{49})^{3/2}$, and since $\beta_1 \beta_2 < 1$ whenever $\beta_1 > 1$, this yields 
\[
\prod_{j=1}^{k} \beta_{j} < \beta_{1} \Big{(}\frac{48}{49}\Big{)}^{k/2} < 2 \Big{(}\frac{48}{49}\Big{)}^{k/2}. \qedhere
\]
\end{proof}

%
%
%
%
%
%
%
%


Since $n_{k+1} \leq 2m_{k+1} + 1 = 2b_{k+1} + 1$, we have $\frac{(n_{k+1}+1)d_{k}}{d_{k+1}} \leq \frac{(2b_{k+1}+2)d_{k}}{b_{k+1}d_{k}} = 2 + \frac{2}{b_{k+1}} \leq 4$, and so
\[
\frac{n_{k+1}\len{\pi(0)}\prod_{i=1}^k (n_i - m_i)}{d_{k+1}}
= \frac{n_{k+1}d_{k}}{d_{k+1}}\frac{\len{\pi(0)}a_{1} \cdots a_{k+1}}{d_{k}}
\leq 4 \prod_{j=1}^{k} \beta_{j} < 8\left(\frac{48}{49}\right)^{k/2}.
\]
Defining $\epsilon_{k} := 8(\frac{48}{49})^{k/2}$ completes the proof.
\end{proof}

\begin{proof}[Proof of Theorem \ref{disc}]
Our technique for verifying discrete spectrum of $X$ is by using mean almost periodicity, which requires a definition. The \textbf{upper density} of $A \subset \mathbb{N}$, denoted $\overline{d}(A)$, is
$\limsup \frac{|A \cap \{1, \ldots, n\}|}{n}$. It's easy to check that upper density is subadditive, i.e.
$\dens(A \cup B) \leq \dens(A) + \dens(B)$ for every $A, B$.

A subshift $X$ is \textbf{mean almost periodic} if for all $\epsilon > 0$ and all $x \in X$, there exists a syndetic set $S$ so that for all $s \in S$, $x$ and $\sigma^s x$ differ on a set of locations with upper density less than $\epsilon$. It is well-known that mean almost periodicity implies discrete spectrum; see for instance Theorem 2.8 of \cite{MR2569181}.

Examples of aperiodic but mean almost periodic subshifts are given by the Sturmian subshifts and also so-called regular Toeplitz subshifts. Since our hypotheses are satisfied by Sturmian subshifts, their mean almost periodicity follows as a corollary of our proof.


By Proposition~\ref{decomp}, $X$ is the orbit closure of 
\[
x^{(m_k), (n_k)} = \lim_{k \rightarrow \infty} (\pi \circ \tau_{m_1, n_1} \circ \tau_{m_2, n_2} \circ \cdots \tau_{m_k, n_k})(0)
= \lim_{k \rightarrow \infty} (\pi \circ \rho_{k})(0)
\]
for some $\pi: \{0,1\} \rightarrow \mathcal{A}^*$ where $\pi(0), \pi(1)$ begin with different letters and 
$|\pi(0)| < |\pi(1)| < 2|\pi(0)|$ and some sequences $(m_k), (n_k)$ satisfying $0 < m_k < n_k \leq 2m_k$ or $(m_{k},n_{k}) = (1,3)$.

We again use the notations
$a_{k+1} = n_k - m_k$ and $d_k = |(\pi \circ \rho_k)(0)|$ as in the proof of Proposition~\ref{error}.

For any $k > 0$ and $p \in \mathbb{N}$, define the words
\begin{align*}
y_{0,k,p} = ((\pi \circ \rho_k)(0))^p (\pi \circ \rho_k)(1), \ & z_{0,k,p} = (\pi \circ \rho_k)(1) ( (\pi \circ \rho_k)(0))^p, \\
y_{1,k,p} = ((\pi \circ \rho_k)(1))^p (\pi \circ \rho_k)(0), \ & z_{1,k,p} = (\pi \circ \rho_k)(0) ( (\pi \circ \rho_k)(1))^p.\
\end{align*}

We will prove the following by induction:
\begin{equation}\label{diffbd}
\textrm{$y_{i,k,p}$, $z_{i,k,p}$ differ on fewer than $2|\pi(1)| p a_1 \ldots a_{k+1}$ locations ($i \in \{0,1\}$)}.
\end{equation}

The base case $k = 0$ trivially holds, since the lengths of $y_{0,0,p}, z_{0,0,p}, y_{1,0,p}, z_{1,0,p}$ are less than $2p|\pi(1)|$. 

Assume now that (\ref{diffbd}) holds for some $k-1$ (and all $p$). 

Consider first the case when $n_{k} \leq 2m_{k}$.

Then by definition of $\tau_{m_k, n_k}$, if we write $u = (\pi \circ \rho_{k-1})(1)$, $v = (\pi \circ \rho_{k-1})(0)$, $m = m_k$, and $n = n_k$, then $y_{0,k,p} = (v^{m-1} u)^p v^{n-1} u$ and $z_{0,k,p} = v^{n-1} u (v^{m-1} u)^p$. 

Since $v$ is a suffix of $u$, write $u = u^{\prime}v$.  Then, using that $m < n \leq 2m$,
\begin{align*}
y_{0,k,p} &= (v^{m-1}u)^{p} v^{n-1}u
= (v^{m-1}u^{\prime}v)^{p} v^{m-1}v^{n-m}u
= v^{m-1}(u^{\prime}v^{m})^{p}v^{n-m}u \\
&= v^{m-1}(u^{\prime}v^{n-m}v^{2m-n})^{p}v^{n-m}u, \\
z_{0,k,p} &= v^{n-1}u(v^{m-1}u)^{p}
= v^{m-1}v^{n-m}(u^{\prime}v^{m})^{p}u 
= v^{m-1}v^{n-m}(u^{\prime}v^{2m-n}v^{n-m})^{p}u \\
&= v^{m-1}(v^{n-m}u^{\prime}v^{2m-n})^{p}v^{n-m}u.
\end{align*}
Since $\len{u^{\prime}v^{n-m}} = \len{v^{n-m}u^{\prime}}$, this means $y_{0,k,p}$ and $z_{0,k,p}$ differ at a number of locations equal to $p$ times the number of locations where $u^{\prime}v^{n-m}$ and $v^{n-m}u^{\prime}$ differ.  Clearly $u^{\prime}v^{n-m}$ and $v^{n-m}u^{\prime}$ differ on the same number of locations as $u^{\prime}v^{n-m}v = uv^{n-m}$ and $v^{n-m}u^{\prime}v = v^{n-m}u$ differ.  Since $uv^{n-m} = z_{0,k-1,n-m}$ and $v^{n-m}u = y_{0,k-1,n-m}$, the inductive hypothesis gives that they differ on fewer than $2|\pi(1)|(n-m)a_{1}\cdots a_{k}$ locations.  Then $y_{0,k,p}$ and $z_{0,k,p}$ differ on fewer than $2|\pi(1)|p(n-m)a_{1}\cdots a_{k}$ locations.  Since $a_{k+1} = n - m$, this proves the claim.
Similarly,
\begin{align*}
y_{1,k,p} &= (v^{n-1}u)^{p}v^{m-1}u
= v^{n-1}(u^{\prime}v^{n})^{p-1}u^{\prime}v^{m}u \\
&= v^{m-1}v^{n-m}(u^{\prime}v^{m}v^{n-m})^{p-1}u^{\prime}v^{m}u
= v^{m-1}(v^{n-m}u^{\prime}v^{m})^{p}u, \\
z_{1,k,p} &= v^{m-1}u(v^{n-1}u)^{p}
= v^{m-1}(u^{\prime}v^{n})^{p}u 
= v^{m-1}(u^{\prime}v^{n-m}v^{m})^{p}u.
\end{align*}
so $y_{1,k,p}$ and $z_{1,k,p}$ differ on fewer than $2|\pi(1)|pa_{1}\cdots a_{k+1}$ locations.

Consider now the case when $n_{k} = 3$ and $m_{k} = 1$.  Here
\[
(\pi \circ \rho_{k})(0) = (\pi \circ \rho_{k-1})(1), (\pi \circ \rho_{k})(1) = ((\pi \circ \rho_{k-1})(0))^{2} (\pi \circ \rho_{k-1})(1)
\]

By Proposition \ref{decomp}, $n_{k-1} = m_{k-1} + 1$ so we have $(\pi \circ \rho_{k-1})(1) = (\pi \circ \rho_{k-2})(0) (\pi \circ \rho_{k-1})(0)$.

First consider when $m_{k-1} > 1$.  Here $(\pi \circ \rho_{k-2})(0)$ is a prefix of $(\pi \circ \rho_{k-1})(0)$ so there are words $g = (\pi \circ \rho_{k-2})(0)$ and $h$ such that $(\pi \circ \rho_{k-1})(0) = gh$ and $(\pi \circ \rho_{k-1})(1) = ggh$.  Then
$
(\pi \circ \rho_{k})(0) = ggh$ and $(\pi \circ \rho_{k})(1) = (gh)^{2}ggh$ so
\begin{align*}
y_{0,k,p} &= (ggh)^{p} (ghghggh)
= ggh (ggh)^{p-1} ghghggh \\
z_{0,k,p} &= (ghghggh) (ggh)^{p}
= ghg (hgg)^{p-1} hgghggh
\end{align*}
which differ on two pairs of $gh$ and $hg$ and on $p-1$ pairs of $ggh$ and $hgg$.

Our inductive hypothesis does apply directly to $gh$ and $hg$, however
 $gh$ and $hg$ differ on the same number of letters as $ggh = (\pi \circ \rho_{k-2})(0) ((\pi \circ \rho_{k-2})(0))^{m_{k-1}-1} (\pi \circ \rho_{k-2})(1)$ and $ghg = ((\pi \circ \rho_{k-2})(0))^{m_{k-1}-1} (\pi \circ \rho_{k-2})(1) (\pi \circ \rho_{k-2})(0)$.  Those words differ on the same number of letters as $(\pi \circ \rho_{k-2})(0)(\pi \circ \rho_{k-2})(1)$ and $(\pi \circ \rho_{k-2})(1)(\pi \circ \rho_{k-2})(0)$, and by hypothesis they differ on fewer than $2\len{\pi(1)}a_{1}\cdots a_{k-1}$ locations.

Similarly, $gggh = ((\pi \circ \rho_{k-2})(0))^{m_{k-1}+1}(\pi \circ \rho_{k-2})(1)$ and $ghgg = ((\pi \circ \rho_{k-2})(0))^{m_{k-1}-1}(\pi \circ \rho_{k-2})(1)((\pi \circ \rho_{k-2})(0))^{2}$ differ on the same number of letters as $(\pi \circ \rho_{k-2})(1)((\pi \circ \rho_{k-2})(0))^{2}$ and $((\pi \circ \rho_{k-2})(0))^{2}(\pi \circ \rho_{k-2})(1)$ which by hypothesis is fewer than $2\len{\pi(1)}2a_{1}\cdots a_{k-1}$ locations.

Therefore $y_{0,k,p}$ and $z_{0,k,p}$ differ on fewer than $2 \cdot 2\len{\pi(1)}a_{1}\cdots a_{k-1} + 2(p-1)2\len{\pi(1)}a_{1}\cdots a_{k-1}$ locations.  Since $a_{k} = 1$ and $a_{k+1} = 2$, they differ on fewer than $2\len{\pi(1)}p a_{1} \cdots a_{k+1}$ locations.
Similarly,
\begin{align*}
y_{1,k,p} &= (ghghggh)^{p}ggh = ghg(hgghghg)^{p-1} hgghggh \\
z_{1,k,p} &= ggh(ghghggh)^{p} = ggh (ghghggh)^{p-1} ghghggh
\end{align*}
differ on two pairs of $gh$ and $hg$ and on $p-1$ pairs of $hgghghg$ and $ghghggh$.  As $hgghghg$ and $ghghggh$ differ on two pairs of $gh$ and $hg$, the total number of differences is $2p$ times the number of differences between $gh$ and $hg$.  Since $gh$ and $hg$ differ on fewer than $2\len{\pi(1)} a_{1} \cdots a_{k-1}$ locations, and since $a_{k} = 1$ and $a_{k+1} = 2$, $y_{1,k,p}$ and $z_{1,k,p}$ differ on fewer than $2\len{\pi(1)}p a_{1} \cdots a_{k+1}$ locations.

Now consider when $m_{k-1} = 1$.  Here $(\pi \circ \rho_{k-1})(0) = (\pi \circ \rho_{k-2})(1)$ so $(\pi \circ \rho_{k-2})(0)$ is a suffix of $(\pi \circ \rho_{k-1})(1)$.  So there are words $g = (\pi \circ \rho_{k-2})(0)$ and $h$ such that $(\pi \circ \rho_{k-1})(0) = hg$.  Then
$(\pi \circ \rho_{k})(0) = ghg$ and $(\pi \circ \rho_{k})(1) = (hg)^{2}ghg$ so
\begin{align*}
y_{0,k,p} &= (ghg)^{p}hghgghg = gh(ggh)^{p-1}ghghgghg \\
z_{0,k,p} &= hghgghg(ghg)^{p} = hg(hgg)^{p-1}hgghgghg
\end{align*}
which differ on two pairs of $gh$ and $hg$ and on $p-1$ pairs of $ggh$ and $hgg$.  Since $gghg = (\pi \circ \rho_{k-2})(0)^{2} (\pi \circ \rho_{k-1})(0) = (\pi \circ \rho_{k-2})(0)^{2} (\pi \circ \rho_{k-2})(1)$ and $hggg = (\pi \circ \rho_{k-2})(1) ((\pi \circ \rho_{k-2})(0))^{2}$, by hypothesis they differ on fewer than $2\len{\pi(1)}2a_{1}\cdots a_{k-1}$ locations.
Then, as above, $y_{0,k,p}$ and $z_{0,k,p}$ differ on fewer than $2\len{\pi(1)}p a_{1} \cdots a_{k+1}$ locations.  Similarly,
\begin{align*}
y_{1,k,p} &= (hghgghg)^{p}ghg = hg (hgghghg)^{p-1} hgghgghg \\
z_{1,k,p} &= ghg(hghgghg)^{p} = gh(ghghggh)^{p-1} ghghgghg
\end{align*}
differ on $2p$ pairs of $gh$ and $hg$ so $y_{1,k,p}$ and $z_{1,k,p}$ differ on fewer than $2\len{\pi(1)} p a_{1} \cdots a_{k+1}$ locations.

We will now prove that $X$ is mean almost periodic. Fix any $k$, and as before, define $u = (\pi \circ \rho_{k-1})(1)$, $v = (\pi \circ \rho_{k-1})(0)$, $m = m_k$, and $n = n_k$. Choose any $y \in X$; by minimality of $X$, $y$ can be written as a bi-infinite concatenation of the words $(\pi \circ \rho_k)(0) = v^{m-1} u$ and $(\pi \circ \rho_k)(1) = v^{n-1} u$. We may assume without loss of generality that $y$ contains $v^{m-1} u$ starting at the origin, since any syndetic set $S$ as in the definition of mean almost periodicity for $y$ also works for any shift of $y$. Since $d_{k} = |v^{m-1}u|$, let us write
\begin{align*}
y                 & = \ldots.v^{m-1} u v^{i_1 - 1} u v^{i_2 - 1} u \ldots\\
\sigma^{d_k} y & = \ldots.v^{i_1 - 1} u v^{i_2 - 1} u \ldots
\end{align*}
where each $i_k$ is either $m$ or $n$.
We can rewrite as
\begin{align*}
y                 & = \ldots.v^{m-1} (u v^{i_1 - m}) v^{m-1} (u v^{i_2 - m}) v^{m-1} \ldots\\
\sigma^{d_k} y & = \ldots.v^{m-1} (v^{i_1 - m} u) v^{m-1} (v^{i_2 - m} u) v^{m-1} \ldots
\end{align*}

The words inside parentheses are unequal exactly when $i_j = n$, in which case they are the pair $u v^{n-m}$, $v^{n-m} u$.
Since the lengths of $u v^{n-m}$ and $v^{n-m} u$ are the same, this means that the only differences in $y$ and $\sigma^{d_k} y$ occur within pairs $u v^{n-m}$, $v^{n-m} u$. By (\ref{diffbd}), the number of differences in any such pair is bounded from above by $2|\pi(1)| (n-m) a_1 \ldots a_k = 2|\pi(1)| a_1 \ldots a_{k+1}$. When $y$ is partitioned into its level-$(k+1)$ words $(\pi \circ \rho_{k+1})(0)$ and $(\pi \circ \rho_{k+1})(1)$ (and $\sigma^{d_k}$ is partitioned at the same locations), each partitioned segment contains exactly one such pair $u v^{n-m}$, $v^{n-m} u$. Since each such segment has length at least $|(\pi \circ \rho_{k+1})(0)| = d_{k+1}$, 
\[
\dens \left( \{t \ : \ y(t) \neq (\sigma^{d_k} y)(t)\} \right) \leq \frac{2|\pi(1)| a_1 \ldots a_{k+1}}{d_{k+1}}.
\]
For ease of notation, we define $D_q = \{t \ : \ y(t) \neq y(t+q)\}$ for every $q$; by the above,
\begin{equation}\label{diffdens}
\dens(D_{d_k}) \leq \frac{2|\pi(1)| a_1 \ldots a_k a_{k+1}}{d_{k+1}}.
\end{equation}
Now, fix any $k$ and consider the set
\[
S_k := \left\{\sum_{i = k}^{r} p_i d_i \ : \ r > k, 0 \leq p_i \leq n_{i+1}+1\right\}.
\]
We claim that $S_k$ is syndetic. To see this, note that $n_{i+1} d_i > d_{i+1}$ for all $i$ since $d_{i+1} = m_{i+1}d_{i} + (n_{i} - m_{i})d_{i-1} \leq m_{i+1}d_{i} + (m_{i}+1)d_{i-1} < m_{i+1}d_{i} + d_{i} + d_{i-1} \leq (m_{i+1}+2)d_{i} \leq (n_{i+1}+1)d_{i}$, and so a simple greedy algorithm shows that for all $M \in \mathbb{N}$, there exists $s \in S_k$ with $M \leq s < M + d_k$.

Finally, choose any $s = \sum_{i = k}^{r} p_i d_i \in S_k$. For any $\ell_1, \ell_2 \in \mathbb{N}$,
$D_{\ell_1 + \ell_2} \subset D_{\ell_1} \cup (D_{\ell_2} - \ell_1)$ since $t \in D_{\ell_{1}+\ell_{2}}$ implies at least one of $y(t) \ne y(t+\ell_{1})$ or $y(t+\ell_{1}) \ne y(t+\ell_{1}+\ell_{2})$ , and so $\dens(D_{\ell_1 + \ell_2}) \leq \dens(D_{\ell_1}) + \dens(D_{\ell_2})$. Using this repeatedly implies
\[
\dens(D_s) = \dens\left(D_{\sum_{i = k}^{r} p_i d_i}\right) \leq \sum_{i = k}^r p_i \dens(D_{d_i}) \leq \sum_{i = k}^r
\frac{2|\pi(1)| n_{i+1} a_1 \ldots a_{i+1}}{d_{i+1}}.
\]
Proposition \ref{error} implies that $\frac{(n_{i+1}+1) \len{\pi(0)} a_{1}\cdots a_{i+1}}{d_{i+1}} < \epsilon_i$ for a sequence $\epsilon_i$ which is exponentially decaying.  Then
$
\dens(D_s) < \sum_{i=k}^r \frac{2\len{\pi(1)}}{\len{\pi(0)}} \epsilon_i
$.
Since $(\epsilon_i)$ is summable, the right-hand side becomes arbitrarily small as $k \rightarrow \infty$, and so $X$ is mean almost periodic, and therefore has discrete spectrum.
\end{proof}

\begin{remark}\label{alpha}
We remark that in fact this proof yields an explicit formula for an eigenvalue of $X$. Namely, define a sequence
$(c_k)$ by $c_{-1} = 1$, $c_0 = 0$, and the same recursion $c_{k+1} = b_{k+1} c_k + a_{k+1} c_{k-1}$. Basic continued fraction theory implies that $\frac{c_k}{d_k}$ approaches a limit $\alpha$, and that for all $k$,
\[
\left| \frac{c_k}{d_k} - \alpha \right| < 
\left| \frac{c_k}{d_k} - \frac{c_{k+1}}{d_{k+1}} \right| = 
\frac{|\pi(0)|a_1 \ldots a_{k+1}}{d_k d_{k+1}} = 
\frac{|\pi(0)| \prod_{i=1}^k (n_i - m_i)}{d_k d_{k+1}}.
\]
Therefore, the distance from $d_k \alpha$ to the nearest integer is less than $\frac{|\pi(0)| \prod_{i=1}^k (n_i - m_i)}{d_{k+1}}$, which decays exponentially by Proposition~\ref{error}. If we define $\lambda = e^{2\pi i \alpha}$, then $\lambda^{d_k}
= \lambda^{|(\pi \circ \rho_k)(0)|}$ approaches $1$ with exponential rate. By definition,
$|(\pi \circ \rho_k)(1)| = d_k + (n_k - m_k) d_{k-1}$.
The distance from $(n_k - m_k) d_{k-1} \alpha$ to the nearest integer is less than 
$\frac{n_k |\pi(0)| \prod_{i=1}^{k-1} (n_i - m_i)}{d_k}$, which again decays exponentially by Proposition~\ref{error}. Therefore, 
$\lambda^{|(\pi \circ \rho_k)(1)|}$ approaches $1$ with exponential rate as well.

From this, an essentially identical argument to that of Host from \cite{MR873430} (see also p. 170-171 from \cite{MR2590264}) shows that $\lambda$ is an eigenvalue (in fact a continuous one). (His argument was for a single substitution $\tau$, but the construction can be done virtually without change with $\tau^k$ replaced by $\pi \circ \rho_k$.)

We can even represent $\alpha$ (and therefore $\lambda$) in terms of generalized continued fractions. If we defined an alternate sequence $(e_k)$ by the same recursion with $e_{-1} = 0$ and $e_0 = 1$, then $\frac{c_k}{e_k}$ is just the $k$th convergent to the generalized continued fraction
\[
\beta = \cfrac{a_1}{b_1+\cfrac{a_2}{b_2+\cfrac{a_3}{b_3+\ddots}}} = \cfrac{1}{m_1+\cfrac{n_1-m_1}{m_2+\cfrac{n_2-m_2}{m_3+\ddots}}}.
\]
In particular, $\frac{c_k}{e_k} \rightarrow \beta$.  Since $c_{-1}=1, c_{0}=0, e_{-1}=0, e_{1}=1$ and $c_{k}$, $d_{k}$ and $e_{k}$ are all defined by the same (linear) recursion, $d_{k} = d_{-1}c_{k} + d_{0}e_{k}$ for all $k$.  Then, as $d_{-1} = \len{\pi(1)} - \len{\pi(0)}$ and $d_{0} = \len{\pi(0)}$,
\[
\alpha = \lim \frac{c_{k}}{d_{k}} = \lim \Big{(}d_{-1} + d_{0}\Big{(}\frac{e_{k}}{c_{k}} - 1\Big{)}\Big{)}^{-1} = (d_{-1} + d_{0} (\beta^{-1} - 1))^{-1} = \frac{\beta}{\len{\pi(1)}\beta + \len{\pi(0)}(1 - \beta)}
\]
%
Therefore, the eigenvalue $\lambda$ can be written as $\exp\left(2\pi i \left(\frac{\beta}{|\pi(1)| \beta + |\pi(0)|(1-\beta)}\right)\right)$.
\end{remark}

\section{A weak mixing subshift with \texorpdfstring{$C = \nicefrac{3}{2}$}{C = 3/2}}

\begin{theorem}\label{example}
There exists an infinite transitive subshift $X$ which is uniquely ergodic, has unique measure which is weak mixing, and 
for which $\limsup \frac{p(q)}{q} = \frac{3}{2}$.
\end{theorem}

The complexity estimates in Theorem \ref{example} will follow from a general formula for word complexity of subshifts with the structure from Proposition \ref{decomp}, which may be of independent interest.




\begin{proposition}\label{sum2}
Let $X$ be the orbit closure of $x^{(m_{k}),(n_{k})}$ for $\pi$ and $(\tau_{m_{k},n_{k}})$ satisfying the conclusions of Proposition \ref{decomp}. 
Then there exists a constant $K$ such that for $k \geq 2$,
\begin{align*}
p(q) &= \left\{ \begin{array}{ll} q + \sum_{j=2}^{k} (n_{j} - m_{j} - 1)\len{v_{j}} + K & \text{if}~\len{s_{k}v_{k}^{n_{k}-2}p_{k}} \leq q \leq \len{s_{k+1}v_{k+1}^{m_{k+1}-1}p_{k+1}} \\
2q - \len{s_{k}v_{k}^{m_{k}-1}p_{k}} + \sum_{j=2}^{k-1} (n_{j} - m_{j} - 1)\len{v_{j}} + K & \text{if}~\len{s_{k}v_{k}^{m_{k}-1}p_{k}} \leq q \leq \len{s_{k}v_{k}^{n_{k}-2}p_{k}}. \end{array} \right.
\end{align*}
\end{proposition}

\begin{proof}
We claim first that the words $p_{\infty} := \lim s_{k}p_{k} = \lim s_{1}v_{2} \cdots v_{k} v_{k}^{m_{k}-1}v_{k-1}^{m_{k-1}-1} \cdots v_{1}^{m_{1}-1}$ and $s_{k}v_{k}^{n_{k}-2}p_{k}$ for $k > 0$ are right-special.

Since $v_{k+1} = v_{k}^{m_{k}-1}u_{k}$ and $u_{k+1} = v_{k}^{n_{k}-1}u_{k}$, $p_{k+1} = v_{k}^{m_{k}-1}p_{k}$.  By induction then $p_{k+1} = v_{k}^{m_{k}-1}\cdots v_{1}^{m_{1}-1}$ as $p_{1}$ is empty.  By Lemma \ref{s2}, $s_{k}p_{k}$ is a suffix of $s_{k+1}p_{k+1} = s_{k}v_{k}^{m_{k}-1}u_{k}v_{k}^{m_{k}-1}p_{k}$.  As $\len{s_{k+1}} > \len{s_{k}}$, this shows $p_{\infty}$ exists and is left-infinite.
 
By definition of $p_{k}$ as the maximal common prefix, $p_{k}\pi(0)$ and $p_{k}\pi(1)$ are both in the language since each of $u_{k}$ and $v_{k}$ must have one of them as a prefix and they cannot have the same one.  So $p_{k}$ is right-special for each $k$ (as $\pi(0)$ and $\pi(1)$ begin with different letters) hence $p_{\infty}$ is right-special.  That $s_{k}v_{k}^{n_{k}-2}p_{k}$ is right-special follows from Lemma \ref{induction}.





Next we claim that every right-special word is a suffix of $p_{\infty}$ or of $s_{k}v_{k}^{n_{k}-2}p_{k}$ for some $k > 0$.

Since every right-special word of length at least $\len{s_{1}v_{1}^{m_{1}-1}}$ is a suffix of a concatenation of $u_{1}$ and $v_{1}$, any right-special word with $s_{2}p_{2} = s_{1}v_{1}^{m_{1}-1}u_{1}v_{1}^{m_{1}-1}$ as a suffix is of the form $xu_{1}v_{1}^{m_{1}-1}$ where $x$ is a suffix of a concatenation of $u_{1}$ and $v_{1}$.  If $x$ were not a suffix of a concatenation of $v_{2}$ and $u_{2}$ then $u_{1}v_{1}^{r}u_{1}$ for $r \ne m_{1}-1, n_{1}-1$ must appear somewhere in $x$ but this is impossible 
by definition of $\tau_{m_1, n_1}$.
So every right-special word with $s_{2}p_{2}$ as a suffix is of the form $xp_{2}$ where $x$ is a suffix of a concatenation of $v_{2}$ and $u_{2}$.

Assume that any word with $s_{k}p_{k}$ as a suffix is necessarily of the form $xp_{k}$ where $x$ is a concatenation  of $u_{k}$ and $v_{k}$.  Let $w$ be a word which has $s_{k+1}p_{k+1}$ as a suffix.  Since $s_{k+1}p_{k+1} = s_{k}v_{k+1}v_{k}^{m_{k}-1}p_{k}$ which has $s_{k}p_{k}$ as a suffix, $w = xv_{k+1}v_{k}^{m_{k}-1}p_{k}$ where $x$ is a suffix of a concatenation of $u_{k}$ and $v_{k}$.  If $x$ were not a suffix of a concatenation of $u_{k+1}$ and $v_{k+1}$ then somewhere in $xv_{k+}$ there must appear $u_{k}v_{k}^{r}u_{k}$ for $r \ne n_{k}-1, m_{k}-1$ or $v_{k}^{t}$ for $t > n_{k}-1$.  But this is impossible by definition of $\tau_{m_k, n_k}$.
By induction, then for all $k$, any word with suffix $s_{k}p_{k}$ is of the form $xp_{k}$ where $x$ is a suffix of a concatenation of $u_{k}$ and $v_{k}$.

Since $v_{k}$ is a suffix of $u_{k}$ for $k > 1$, write $u_{k} = u_{k}^{\prime}v_{k}^{\ell_{k}}$ for $\ell_{k} \geq 1$ maximal.  Note that $s_{k}$ has $v_{k}^{\ell_{k}}$ as a suffix.

Let $w$ be a right-special word with $\len{w} \geq \len{s_{1}p_{1}}$.  Take $k \geq 1$ maximal so that $w$ has $s_{k}p_{k}$ as a suffix.  By the above, $w = xp_{k}$ is where $x$ is a suffix of a concatenation of $u_{k}$ and $v_{k}$ in any left-infinite word.  Choose $t \geq 0$ maximal so that $v_{k}^{t}p_{k}$ is a suffix of $w$.

Suppose $w$ is not a suffix of $s_{k}v_{k}^{t-\ell_{k}}p_{k}$.  Then $u_{k}v_{k}^{t-\ell_{k}}p_{k}$ must be right-special since all letters to the left of $s_{k}$ are forced to come from $u_{k}$ by maximality of $\ell_{k}$.  As the $p_{k}$ must appear as a prefix of both $v_{k}$ and $u_{k}$, then $u_{k}v_{k}^{t-\ell_{k}}u_{k}$ and $u_{k}v_{k}^{t-\ell_{k}}v_{k}$ are in the language so $t - \ell_{k} = m_{k} - 1$.  But then $w$ has $v_{k}^{m_{k}-1}p_{k} = p_{k+1}$ as a suffix, contradicting the maximality of $k$.

So $w$ is a suffix of $s_{k}v_{k}^{t - \ell_{k}}p_{k}$.
Suppose $t - \ell_{k} \geq n_{k} - 1$.  Then $w$ has $v_{k}^{n_{k}-1+\ell_{k}}p_{k}$ as a suffix.   As $w$ is right-special, this requires $v_{k}^{n_{k}-1+\ell_{k}}v_{k}$ be in the language.  But that is only possible if $u_{k}$ has $v_{k}^{\ell_{k}+1}$ as a suffix, contradicting the maximality of $\ell_{k}$.

So $t - \ell_{k} \leq n_{k} - 2$.  Then $w$, being a suffix of $s_{k}v_{k}^{t-\ell_{k}}p_{k}$, is a suffix of $s_{k}v_{k}^{n_{k}-2}p_{k}$.
%

Finally, we establish the complexity function is as claimed.
Since $p_{\infty}$ has $s_{k+1}p_{k+1} = s_{k}v_{k+1}v_{k}^{m_{k}-1}p_{k}$ as a suffix, by Lemma \ref{s2}, it has $s_{k}u_{k}v_{k}^{m_{k}-1}p_{k}$ as a suffix.  By Lemma \ref{s3}, the maximal common suffix of $p_{\infty}$ and $s_{k}v_{k}^{n_{k}-2}p_{k}$ is then $s_{k}v_{k}^{m_{k}-1}p_{k}$.  Likewise the maximal common suffix of $s_{k}v_{k}^{n_{k}-2}p_{k}$ and $s_{k^{\prime}}v_{k^{\prime}}^{n_{k^{\prime}}-2}p_{k^{\prime}}$ for $k^{\prime} > k$ is $s_{k}v_{k}^{m_{k}-1}p_{k}$ as $v_{k+1}$ has $u_{k}$ as a suffix.  Therefore each $s_{k}v_{k}^{n_{k}-2}p_{k}$ provides $(n_{k}-2 - (m_{k}-1))\len{v_{k}}$ right-special words (with lengths in $(\len{s_k v_k^{m_k-1} p_k}, \len{s_{k}v_{k}^{n_{k}-2}p_{k}}]$) which are not suffixes of $p_{\infty}$.  Set $K = p(\len{s_{2}p_{2}}) - \len{s_{2}p_{2}}$ and the claim follows.
\end{proof}

\begin{proof}[Proof of Theorem \ref{example}]
Define any increasing $(n_k)$, $(m_k)$ so that $n_k = 2m_k$ for all $k$, $m_1 = 1$, and the sum $\sum_k (n_k)^{-1} < \infty$. Then define $\pi$ to be the identity, define $\tau_{m_k, n_k}$, $\rho_k, a_k, b_k, c_k, d_k$ as in the proof of Proposition \ref{error}, and note that $a_{k+1} = n_k - m_k = m_k = b_k$ for all $k$ and 
$\sum_k (b_k)^{-1} < \infty$. Just as before, $d_k = |\rho_k(0)|$ for all $k$, 
and we wish to impose the additional condition that $d_k$ is prime for every $k>1$. This is easily achieved via induction. First, 
$d_0 = d_1 = 1$, so $d_2 = b_2 d_1 + a_2 d_0 = m_2 + 1$, which can clearly be chosen prime. Then, assume that $d_k$ is prime, and recall
that $d_{k+1} = b_{k+1}d_k+a_{k+1}d_{k-1}$. Both $a_{k+1} = b_k$ and $d_{k-1}$ are positive and less than the prime $d_k$ (since $d_k = b_k d_{k-1} + a_k d_{k-2}$ and $d_{k-2}$ is positive for $k > 1$), meaning that 
$d_k$ and $a_{k+1}d_{k-1}$
are positive and coprime. Then by Dirichlet's theorem, there exist infinitely many choices of
$b_{k+1}$ so that $d_{k+1}$ is prime. As long as the sequence $(b_k)$ is chosen large enough at
each step, we will maintain the condition $\sum_k (b_k)^{-1} < \infty$.

Let $X$ be the orbit closure of $x^{(m_{k}),(n_{k})}$.
$X$ is minimal by construction so by \cite{boshernitzan}, $X$ is uniquely ergodic with unique measure $\mu$.

Suppose for a contradiction that $X$ is not weak mixing, and so there is an eigenvalue $\lambda \neq 1$ with measurable eigenfunction $f$. Our method is again based on the Host's arguments from \cite{MR873430}, where he showed that the existence of an eigenfunction can be used to obtain Diophantine conditions involving the lengths of substitution words, which can be viewed as heights of Rokhlin towers.

One can define Rokhlin towers by $B_k = [\rho_k(0)]$, $h_k = |\rho_k(0)|$, and $T_k = \bigcup_{j = 0}^{h_k - 1} \sigma^j B_k$; since $m_k, n_k \rightarrow \infty$, $\mu(T_k) \rightarrow 1$.  By Remark \ref{ud}, $X$ is uniquely decomposable so the levels of the towers are disjoint. Then, for each $k$, define 
\[
f_k(x) =\sum_{j=0}^{h_{k}-1}  \frac{1}{\mu(B_k)} \Big{(}\int_{\sigma^j B_k} f \ d\mu\Big{)} \bbone_{\sigma^j B_{k}}(x)
\]
i.e., $f_{k}(x) = (\mu(B_k))^{-1}(\int_{\sigma^j B_k} f \ d\mu)$ for $x \in \sigma^j B_{k}$ and $f_{k}(x) = 0$ for $x \notin T_k$.

By the Lebesgue Differentiation Theorem, as $\mu(T_k) \to 1$ and $\mu(\sigma^{j}B_{k}) \to 0$, $f_{k}$ converge almost everywhere to $f$.

Observe that $\sigma^{d_{k}} = \sigma^{|\rho_{k}(0)|}$ takes every occurrence of $\rho_{k}(0)$ to an occurrence of $\rho_{k}(0)$ except for those which are immediately prior to an occurrence of $\rho_{k}(1)$ in some $\rho_{k+1}(0)$ or $\rho_{k+1}(1)$. Then for all $t > 0$, $\sigma^{d_{k+t}}$ takes all occurrences of $\rho_{k}(0)$ appearing in a $\rho_{k+t}(0)$ to an occurrence of $\rho_{k}(0)$ except possibly for those appearing in a $\rho_{k+t}(0)$ immediately prior to a $\rho_{k+t}(1)$.

Let $\{ i_{k} \}$ be any sequence such that $0 < i_{k} < 0.5(m_{k+1}-1)$.  Then as above, for all $t > 0$,
$\sigma^{i_{k+t}d_{k+t}}$ takes all occurrences of $\rho_{k}(0)$ in a $\rho_{k+t}(0)$ to an occurrence of $\rho_{k}(0)$ except possibly for those appearing in a $\rho_{k+t}(0)$ less than $i_{k+t}$ occurrences before a $\rho_{k+t}(1)$ in some $\rho_{k+t+1}(0)$ or $\rho_{k+t+1}(1)$. We also note that since $n_{k+t+1} = 2m_{k+t+1}$, at least one-third of the $\rho_k(0)$ appearing in any $x \in X$ 
are part of some $\rho_{k+t}(0)$. Therefore, 
\[
\mu(\sigma^{i_{k+t}d_{k+t}}[\rho_{k}(0)] \cap [\rho_{k}(0)]) \geq \frac{m_{k+t}-1-i_{k+t}}{m_{k+t}-1} \left( \frac{1}{3} \mu([\rho_{k}(0)])\right)
\]
so, since $i_{k+t} < 0.5(m_{k+t}-1)$,
\[
\mu(\sigma^{i_{k+t}d_{k+t}}(\sigma^j B_{k}) \cap (\sigma^j B_{k})) > \frac{1}{6}\mu(\sigma^j B_{k}).
\]
Then $f_{k}(\sigma^{i_{k+t}d_{k+t}}x) = f_k(x)$ for a set of measure at least $\frac{1}{6}\mu(T_{k})$.  Since $f_{k} \to f$ almost everywhere and $\mu(T_{k}) \to 1$, there is then a positive measure set such that for any sufficiently small $\epsilon > 0$ and almost every $x$ in the set, there exists $k$ so that for all $t$, $|f(\sigma^{i_{k+t}d_{k+t}}x) - f(x)| < \epsilon$.  Therefore $\lambda^{i_{k}d_{k}} \to 1$.

We will prove that this is impossible. Define $r \in (0,1)$ by $\lambda = e^{2\pi i r}$; then $\langle i_k d_k r \rangle \rightarrow 0$ whenever $0 < i_k < 0.5(m_{k+1} - 1)$, which implies that for large enough $k$ (say $k \geq k_0$), 
$\langle d_k r \rangle < 0.05(m_{k+1} - 1)^{-1}$. Clearly $r$ cannot be rational, since all $d_k$ are $1$ or prime.
Since $5n_{k+1} = 10m_{k+1} < 20(m_{k+1}-1)$, for $k \geq k_0$, $\langle d_k r \rangle < 0.2(n_{k+1})^{-1}$. This implies that for all $k \geq k_0$, there exists $c'_k \in \mathbb{Z}$,
so that $\left|r - \frac{c'_k}{d_k} \right| < 0.2(d_k n_{k+1})^{-1} < 0.2(d_{k+1})^{-1}$. (Recall that
$d_{k+1} = b_{k+1} d_k + a_{k+1} d_{k-1} < 2b_{k+1} d_k = n_{k+1} d_k$.) We will prove the following: for all $k > k_0$,

\begin{equation}\label{recursion}
c'_{k+1} = b_{k+1} c'_k + a_{k+1} c'_{k-1}. 
\end{equation}

Assume that $k > k_0$, and denote the right-hand side of (\ref{recursion}) by $c''_{k+1}$. Then, 
\begin{equation}\label{ckbds}
\left|r - \frac{c'_k}{d_k} \right| < 0.2(d_{k+1})^{-1} \textrm{ and } 
\left|r - \frac{c'_{k-1}}{d_{k-1}} \right| < 0.2(d_k)^{-1},
\end{equation}
and so
\begin{equation}\label{bd1}
\left|d_{k+1} r - c'_k \frac{d_{k+1}}{d_k} \right| < 0.2.
\end{equation}

We can simplify
\begin{equation}\label{bd2}
\left|c'_k \frac{d_{k+1}}{d_k} - c''_{k+1}\right| = \left|c'_k\left(b_{k+1} + \frac{a_{k+1} d_{k-1}}{d_k}\right) - b_{k+1} c'_k - a_{k+1} c'_{k-1}\right| = \left| c'_k a_{k+1} \frac{d_{k-1}}{d_k} - a_{k+1} c'_{k-1} \right|.
\end{equation}

By the second inequality in (\ref{ckbds}),
\begin{equation}\label{bd3}
\left|a_{k+1} d_{k-1} r - a_{k+1} c'_{k-1}\right| < \frac{0.2a_{k+1} d_{k-1}}{d_{k}} = \frac{0.2b_k d_{k-1}}{d_{k}} < 0.2.
\end{equation}

Similarly, by the first inequality in (\ref{ckbds}),
\begin{equation}\label{bd4}
\left|a_{k+1} d_{k-1} r - c'_k a_{k+1} \frac{d_{k-1}}{d_k}\right| < \frac{0.2a_{k+1} d_{k-1}}{d_{k+1}} = \frac{0.2b_k d_{k-1}}{d_{k+1}} < 0.2.
\end{equation}

Therefore, by the triangle inequality and (\ref{bd2})-(\ref{bd4}),
\[
\left|c'_k \frac{d_{k+1}}{d_k} - c''_{k+1}\right| < 0.4.
\]

Combining with (\ref{bd1}) via the triangle inequality yields
\begin{equation}\label{bd5}
\left|d_{k+1} r - c''_{k+1}\right| < 0.6.
\end{equation}
Recall that by definition, 
\begin{equation}\label{bd6}
\left|r - \frac{c'_{k+1}}{d_{k+1}}\right| < 0.2 (d_{k+2})^{-1}, \textrm{ and so }
\left|d_{k+1} r - c'_{k+1}\right| < 0.2 \frac{d_{k+1}}{d_{k+2}} < 0.2.
\end{equation}
Finally, (\ref{bd5}) and (\ref{bd6}) imply that $c'_{k+1} = c''_{k+1}$ (since they are both integers), completing the proof that (\ref{recursion}) holds for $k > k_0$.

Since $r$ is irrational and $\frac{c'_k}{d_k} \rightarrow r$, we may also assume without loss of generality (by increasing $k_0$) that $\frac{c'_{k_0}}{d_{k_0}} \neq 
\frac{c'_{1+k_0}}{d_{1+k_0}}$. Then, it is easily proved by induction that for all $k > k_0$,
\[
\left| \frac{c'_k}{d_k} - \frac{c'_{k+1}}{d_{k+1}} \right| = |c'_{1+k_0} d_{k_0} - c'_{k_0} d_{1+k_0}| \frac{a_{1+k_0} \ldots a_{k+1}}{d_k d_{k+1}}.
\]
We abbreviate $Q = |c'_{1+k_0} d_{k_0} - c'_{k_0} d_{1+k_0}|$, and note that $Q \neq 0$ by the assumption that $\frac{c'_{k_0}}{d_{k_0}} \neq \frac{c_{1+k_0}}{d_{1+k_0}}$. 
We can now bound the distance from above using that $a_{j+1}d_{j-1} \leq d_{j}$:
\begin{equation}\label{finalbd}
\left| \frac{c'_k}{d_k} - \frac{c'_{k+1}}{d_{k+1}} \right| = \frac{Q a_{1+k_0} \ldots a_{k+1}}{d_k d_{k+1}}
= \frac{Q}{d_{k_0 - 1} d_{k+1}} \prod_{j=k_0}^{k} \frac{a_{j+1} d_{j-1}}{d_j} > \frac{Q}{d_{k_0 - 1} d_{k+1}} \prod_{j=k_0}^{\infty} \frac{a_{j+1} d_{j-1}}{d_j}.
\end{equation}
Note that 
\[
\frac{d_j}{a_{j+1} d_{j-1}} = \frac{b_j}{a_{j+1}} + \frac{a_j d_{j-2}}{a_{j+1} d_{j-1}} \leq \frac{b_j}{a_{j+1}} + \frac{b_{j-1} d_{j-2}}{a_{j+1} d_{j-1}} < \frac{b_j + 1}{a_{j+1}} < \frac{b_j}{a_{j+1} - 1} = \frac{b_{j}}{b_{j} - 1} =
 \left(1 - b_j^{-1}\right)^{-1}.
\]
Therefore, the product $\prod_{j=k_0}^{\infty} \frac{a_{j+1} d_{j-1}}{d_j}$ is greater than 
$\prod_{j=k_0}^{\infty} \left(1 - \frac{1}{b_j}\right)$, which converges to a positive limit $L$ by the assumption that $\sum b_k^{-1}<\infty$. Combining with (\ref{finalbd}) yields that there exists a positive constant $K = \frac{Q L}{d_{k_0 - 1}}$ so that for all $k > k_0$, 
\begin{equation}\label{finalbd2}
\left| \frac{c'_k}{d_k} - \frac{c'_{k+1}}{d_{k+1}} \right| > \frac{K}{d_{k+1}}.
\end{equation}
However, recall that $|r - \frac{c_{k}^\prime}{d_{k}}| < 0.2(d_{k}n_{k+1})^{-1}$ meaning $|r d_{k+1} n_{k+1} - c_{k}^{\prime}n_{k+1}| < 0.2$ so $c_{k}^{\prime}n_{k+1}$ is the closest integer to $r d_{k+1} n_{k+1}$.  
Since $\langle 0.25n_{k+1}d_{k}r \rangle \to 0$, this implies there exists $k_{1} > k_{0}$ such that $|r d_{k+1} n_{k+1} - c_{k}^{\prime} n_{k+1}| < 0.5K$.  Then $|r - \frac{c_{k}^{\prime}}{d_{k}}| < 0.5K (n_{k+1}d_{k})^{-1}$.  Since $d_{k+1} < n_{k+1}d_{k}$, then $|r - \frac{c_{k}^{\prime}}{d_{k}}| < 0.5K (d_{k+1})^{-1}$.
Then for $k > k_1$,
\[
\left| r - \frac{c'_k}{d_k} \right| < 0.5K(d_{k+1})^{-1} \textrm{ and } 
\left| r - \frac{c'_{k+1}}{d_{k+1}} \right| < 0.5K(d_{k+2})^{-1} < 0.5K(d_{k+1})^{-1},
\]
which contradicts (\ref{finalbd2}) by the triangle inequality. Therefore, our original assumption is false and $X$ is weak mixing.

It remains only to show that the complexity function satisfies the claimed bounds.
Since $\len{p_1} = 0$ and by Remark \ref{ps}, $p_{k+1} = v_{k}^{m_{k}-1}p_{k}$, we have $\len{p_{k}} = \sum_{j=1}^{k-1} (m_{j}-1)\len{v_{j}}$ and therefore, since $n_{j} - m_{j} = m_{j}$,
\[
\sum_{j=1}^{k}(n_{j} - m_{j} - 1)\len{v_{j}} = \sum_{j=1}^{k}(m_{j}-1)\len{v_{j}} = (m_{k}-1)\len{v_{k}} + \len{p_{k}}.
\]
By Proposition \ref{sum2}, then
\begin{align*}
p(\len{s_{k}v_{k}^{2m_{k}-2}p_{k}}) = \len{s_{k}v_{k}^{2(m_{k}-1)}p_{k}} + (m_{k}-1)\len{v_{k}} + \len{p_{k}} + K
= 1.5\len{s_{k}v_{k}^{2m_{k}-2}p_{k}} - 0.5(\len{s_{k}} - \len{p_{k}}) + K.
\end{align*}
Since $\len{p_{k}} + \len{s_{k}} < 3\len{v_{k}}$ and $m_{k} \to \infty$, $\lim \frac{p(\len{s_{k}v_{k}^{2m_{k}-2}p_{k}})}{\len{s_{k}v_{k}^{2m_{k}-2}p_{k}}} = 1.5$.  Proposition \ref{sum2} implies that
the limsup of $\frac{p(q)}{q}$ 
is achieved along some subsequence of $\len{s_{k}v_{k}^{n_{k}-2}p_{k}}$, so $\limsup \frac{p(q)}{q} = 1.5$.
\end{proof}

\begin{remark}
The examples in Theorem \ref{example} also satisfy
$p(q) - 1.5q \to -\infty$ and $\liminf \frac{p(q)}{q} = 1$.  For any $f(q) \to \infty$, such a subshift exists which also satisfies $p(q) < q + f(q)$ infinitely often.
\end{remark}
\begin{proof}
By Remark \ref{ps}, $s_{k+1} = s_{k}v_{k+1}$, so we have $\len{s_{k}} - \len{p_{k}} \leq \len{s_{k}} - \len{v_{k}} = \len{s_{k-1}} \to \infty$ so $p(q) - 1.5q \to -\infty$.
By Proposition \ref{sum2},
\[
p(\len{v_{k}^{m_{k}-1}p_{k}}) = \len{v_{k}^{m_{k}-1}p_{k}} + \sum_{j=1}^{k-1} (n_{j} - m_{j} - 1)\len{v_{j}} +K = \len{v_{k}^{m_{k}-1}p_{k}} + \len{p_{k}} + K
\]
and $\len{p_{k}} < 3\len{v_{k}}$ so since $m_{k} \to \infty$, $\liminf \frac{p(q)}{q} = 1$.
Now let $f(q) \to \infty$ be arbitrary.  
For all $k$, if $v_{k}$ and $p_{k}$ are given, we can choose $b_{k} = m_k$ large enough so that $f((m_{k}-1)\len{v_{k}}+\len{p_k}) > \len{p_{k}} + K$, which implies that $p(\len{v_{k}^{m_{k}-1}p_{k}}) < \len{v_{k}^{m_{k}-1}p_{k}} + f(\len{v_{k}^{m_{k}-1}p_{k}})$.
\end{proof}


%

\vspace{-2em}
\dbibliography{DiscreteSpectrum}

\end{document}